%% file: SINUM.tex
\documentclass[hidelinks,onefignum,onetabnum]{siamart220329}

\usepackage{lipsum}
\usepackage{amsfonts}
\usepackage{graphicx}
\usepackage{caption}
\usepackage{subcaption}
\usepackage{epstopdf}
\usepackage{algorithmic}
\ifpdf
  \DeclareGraphicsExtensions{.eps,.pdf,.png,.jpg}
\else
  \DeclareGraphicsExtensions{.eps}
\fi
\graphicspath{{figures/SINUM_CFIE/}}

\input{lvc_style}

\newsiamremark{remark}{Remark}
\newsiamremark{hypothesis}{Hypothesis}
\crefname{hypothesis}{Hypothesis}{Hypotheses}
\newsiamthm{claim}{Claim}

\headers{An operator preconditioned CFIE in electromagnetics}{V.~C. Le and K. Cools}

\title{An operator preconditioned combined field integral equation for electromagnetic scattering\thanks{Submitted to the editors DATE.
\funding{This work was funded by the European Research Council (ERC) under the European Union's Horizon 2020 research and innovation programme (Grant agreement No. 101001847).}}}

\author{Van Chien Le\thanks{IDLab, Department of Information Technology, Ghent University - imec, 9000 Ghent, Belgium 
  (\email{vanchien.le@ugent.be}, \email{kristof.cools@ugent.be}).}
\and Kristof Cools\footnotemark[2]}

\ifpdf
\hypersetup{
  pdftitle={An operator preconditioned combined field integral equation for electromagnetic scattering},
  pdfauthor={V.~C. Le and K. Cools}
}
\fi

\begin{document}

\maketitle

\begin{abstract}
This paper aims to address two issues of integral equations for the scattering of time-harmonic electromagnetic waves by a perfect electric conductor with Lipschitz continuous boundary: ill-conditioned boundary element Galerkin matrices on fine meshes and instability at spurious resonant frequencies. The remedy to ill-conditioned matrices is operator preconditioning, and resonant instability is eliminated by means of a combined field integral equation. Exterior traces of single and double layer potentials are complemented by their interior counterparts for a purely imaginary wave number. We derive the corresponding variational formulation in the natural trace space for electromagnetic fields and establish its well-posedness for all wave numbers. A Galerkin discretization scheme is employed using conforming edge boundary elements on dual meshes, which produces well-conditioned discrete linear systems of the variational formulation. Some numerical results are also provided to support the numerical analysis. 
\end{abstract}



\section{Introduction}

\subsection{The scattering boundary value problem}

This paper concerns the scattering of electromagnetic waves by a perfect electric conductor, which plays a fundamental role in computational electromagnetics. Let $\Om \sst \R^3$ be an open bounded domain with a connected Lipschitz boundary $\Gm := \pa\Om$. The exterior region $\Om^c := \R^3 \setminus \ovl{\Om}$ is filled by a homogeneous, isotropic, and linear material with permittivity $\epsilon$ and permeability $\mu$, both are positive constants in $\Om^c$. Electromagnetic waves propagating outside $\Om$ are excited by a time-harmonic incident electric field $\eb^{in}$ of angular frequency $\om > 0$. Therefore, we can switch to a frequency-domain problem with unknown complex-valued spatial functions. The scattered electric field $\eb$ satisfies the following exterior Dirichlet boundary-value problem for the electric wave equation \cite[Section~6.4]{Colton1992}
\begin{align}
    & \curlt \, \curlt \, \eb - \kappa^2 \eb = \zrb && \text{in} \q \Om^c, \label{eq:wave} \\
    & \eb \times \nv = - \eb^{in} \times \nv && \text{on} \q \Gm, \label{eq:DBC}
\end{align}
supplemented with the Silver-M\"{u}ller radiation condition
\begin{equation}
    \label{eq:SilverMuller}
    \lim_{r \to \infty} \int_{\pa B_r} \abs{\curlt \, \eb \times \nv + i\kappa (\nv \times \eb) \times \nv}^2 \ds = 0.
\end{equation}
Here, $\kappa = \om \sqrt{\epsilon\mu} > 0$ is the wave number, $\nv$ is the unit normal vector on $\Gm$ pointing from $\Om$ to $\Om^c$, and $B_r$ is the ball of radius $r > 0$ centered at $0$. We refer the reader to Rellich's lemma \cite{Cessenat1996,Nedelec2001} for the existence and uniqueness of a solution to \eqref{eq:wave}-\eqref{eq:SilverMuller}.

\subsection{Challenges in boundary integral equations}

Boundary integral equations have become the most popular method to solve electromagnetic scattering problems in unbounded domains. Based on the integral representation formulas for solutions to Maxwell's equations, this method poses an alternative problem on the boundary of domains, leading to discrete systems of much smaller size. Prominent examples are the electric and magnetic field integral equations (EFIE and MFIE). This paper aims to address two issues of boundary integral equations for electromagnetic scattering by perfectly conducting bodies with Lipschitz continuous boundary: instability at spurious resonant frequencies (in short, resonant instability) and ill-conditioned Galerkin boundary element matrices on fine meshes. These issues arise when $\kappa^2$ is close to a resonant frequency (the former), or when the discrete problem involves a large number of unknowns (the latter), both manifest themselves in the ill-conditioning of the discrete linear systems.

\subsection{Previous and related works}

Resonant frequencies are Dirichlet or Neumann eigenvalues of the $\curlt \, \curlt$-operator inside $\Om$, at which the standard boundary integral equations are not uniquely solvable. Among some approaches to overcome the resonant instability (see, e.g., \cite{FK1998,HL1996}), combined field integral equations (CFIEs) are vastly more popular than others. CFIEs owe their name to an appropriate combination of the single and double layer potentials. This class of integral equations for electromagnetic scattering was pioneered by Panich in \cite{Panich1965}. A regularized CFIE was then introduced by Kress in \cite{Kress1986}. Both formulations are only applicable for domains with sufficiently smooth boundaries (specifically, belonging to the class $\Cs^2$), which are not suitable for the numerical implementation of the boundary element method. With the advancement in numerical analysis of Maxwell's equations in Lipschitz domains (see, e.g., \cite{BC2001,BC2001b,BCS2002a,BHV+2003}), some coercive CFIEs were proposed by Buffa and Hiptmair in \cite{BH2005}, and by Steinbach and Windisch in \cite{SW2009}, which are applicable for domains with Lipschitz continuous boundaries.

Despite the fact that CFIEs are well-posed at all frequencies, they may produce ill-conditioned matrices when involving a large number of discrete unknowns, which leads the numerical resolution by means of iterative schemes to be extremely expensive. The typical approach to curing this challenge is to employ an algebraic preconditioner. Some preconditioners for CFIEs were presented in \cite{ABL2007,Levadoux2008,BEP+2009,BT2014,AAE2016}. All the treatments primarily rely on the fact that the MFIE operator on a sufficiently smooth surface is a Fredholm operator of second kind. Unfortunately, this special property is no longer valid for non-smooth Lipschitz domains.

In computational engineering, numerous studies have been done on operator preconditioned CFIEs for electromagnetic scattering. For instance, three different CFIEs were introduced in \cite{CDE+2002}, where exterior EFIE and MFIE operators of a purely imaginary wave number serve as preconditioners. The unique solvability and the well-conditioning of Galerkin discretization matrices of these formulations were proven for spheres and surfaces obtained by smooth deformations of a sphere. In spite of their similar numerical behaviors, the CFIE which consists of the EFIE preconditioned by its counterpart and the MFIE (without a preconditioner) has attracted the most attention from engineers because of its computational simplicity. The uniqueness of a solution to this formulation was proven in \cite{ACB+2012}, and some stable and accurate discretization schemes were proposed in \cite{BAC+2009,CAY+2010}. A different approach was introduced in \cite{MBC+2020}, where the exterior EFIE and MFIE operators are complemented by their interior counterparts of a purely imaginary wave number. A proof of the injectivity of the governing operator was provided. Moreover, it has been shown that a Galerkin discretization of this Yukawa-Calder\'{o}n CFIE in conjunction with a discrete Helmholtz decomposition and a rescaling procedure results in linear systems that are well-conditioned at low frequencies and numerical experiments demonstrated well-conditioning for fine meshes. Both the stability at low frequencies and for fine meshes are maintained when multiply connected scatterers are considered and when round-off and quadrature errors are introduced.

\subsection{Novelty and contributions}

These attractive properties justify additional analysis of integral equations of the class introduced in \cite{MBC+2020}. In this paper, we investigate an indirect equation on Lipschitz domains. The indirect formulation is based on a potential representation for the scattered electric field $\eb$. Our main contributions are twofold: proving the well-posedness of the governing CFIE for all wave numbers and proposing a stable Galerkin boundary element discretization. The uniqueness of a solution is established by means of the ellipticity of the EFIE operator with a purely imaginary wave number. A generalized G{\aa}rding inequality is achieved by leveraging a Calder\'{o}n projection formula. In engineering papers, a proof of the coercivity of CFIE operators is typically omitted and in the analysis it is assumed that the MFIE operator is a Fredholm operator of second kind. However, this assumption is only valid for sufficiently smooth domains. In Lipschitz domains, the injectivity and a generalized G{\aa}rding inequality of the governing operator are necessary and sufficient for the unique solvability of the CFIE because of the Fredholm alternative. A mixed Galerkin discretization is then introduced using $\divt_\Gm$-conforming boundary elements on a pair of dual meshes. The unique solvability and the uniform boundedness of condition number of resulting linear systems are shown. It is important to note that the discretization scheme introduced in this paper differs from what practitioners typically use and in particular differs from the scheme studied in \cite{MBC+2020}.

\subsection{Outline}

This paper is organized as follows: the next section provides a concise summary of relevant function spaces and trace spaces, which are needed for numerical analysis throughout the paper. Then, section~\ref{sec:potentials} recalls the crucial potentials and integral operators for electromagnetic scattering. In section 4, we introduce the new CFIE, derive its variational formulation and prove its well-posedness. An equivalent mixed variational formulation and its Galerkin boundary element discretization are proposed in section~\ref{sec:discretization}. Section~\ref{sec:results} is devoted to some numerical results, which corroborate the stability of the CFIE as well as the convergence and well-conditioning of discrete linear systems. We end up with some conclusions and an outlook on future works in section~\ref{sec:conclusions}.

\section{Traces and spaces}
\label{sec:traces}

For any domain $\Om \subseteq \R^3$, let $\Hs^s(\Om)$ and $\HHs^s(\Om)$, with $s \geq 0$, be the Sobolev spaces of complex-valued scalar and vector functions on $\Om$ equipped with the standard graph norms, where $\Hs^0(\Om)$ and $\HHs^0(\Om)$ coincide with the Lebesgue spaces $\Ls^2(\Om)$ and $\LLs^2(\Om)$. For any function $u$, its complex conjugate is denoted by $\ovl{u}$. The following function spaces are the natural spaces for solutions of the electric wave equation \eqref{eq:wave} on bounded domains
\begin{align*}
    \HHs(\curlt, \Om) & := \brac{\ub \in \LLs^2(\Om) : \curlt \, \ub \in \LLs^2(\Om)}, \\
    \HHs(\curlt^2, \Om) & := \brac{\ub \in \HHs(\curlt, \Om) : \curlt \,\curlt \, \ub \in \LLs^2(\Om)},
\end{align*}
which are respectively endowed with the norms
\begin{align*}
    \norm{\ub}^2_{\HHs(\curlt, \Om)} & := \norm{\ub}^2_{\LLs^2(\Om)} + \norm{\curlt \, \ub}^2_{\LLs^2(\Om)}, \\
    \norm{\ub}^2_{\HHs(\curlt^2, \Om)} & := \norm{\ub}^2_{\HHs(\curlt, \Om)} + \norm{\curlt \, \curlt \, \ub}^2_{\LLs^2(\Om)}.
\end{align*}
On unbounded domains $\Om$, the space $\HHs_\textup{loc}(\curlt^2, \Om)$ is defined as the set of all vector functions $\ub$ such that $\varphi \ub \in \HHs(\curlt^2, \Om)$ for all compactly supported smooth scalar functions $\varphi \in \Cs^{\infty}(\R^3)$.

Next, we briefly introduce some trace spaces related to Maxwell's equations in a Lipschitz domain $\Om$. For more details, the reader is referred to the articles \cite{BC2001,BC2001b,BCS2002a,BH2003}. We define the continuous tangential trace operator $\gm_D : \HHs^1(\Om) \to \LLs^2_\tv(\Gm)$ by
\[
     \gm_D : \ub \mapsto \gm(\ub) \times \nv,
\]
where $\gm : \HHs^1(\Om) \to \LLs^2(\Gm)$ is the standard trace operator. The range of $\gm_D$ in $\LLs^2_\tv(\Gm)$ is denoted by $\HHs^{1/2}_{\times}(\Gm)$. Its dual space is denoted by $\HHs^{-1/2}_{\times}(\Gm)$, whose elements are identified via the $\LLs^2_\tv(\Gm)$ anti-symmetric pairing
\[
    \inprod{\ub, \vb}_{\times, \Gm} := \int_\Gm (\ub \times \nv) \cdot \vb \ds, \qqq \ub, \vb \in \LLs^2_\tv(\Gm) :=  \brac{\ub \in \LLs^2(\Gm) : \ub \cdot \nv = 0}.
\]
Let $s \in \{\frac{1}{2}, \frac{3}{2}\}$ and $\Hs^s(\Gm)$ be the trace space of functions in $\Hs^{s+1/2}(\Om)$ (the definition of $\Hs^{3/2}(\Gm)$ can be found in \cite[Proposition~3.4]{BCS2002a}). The dual space of $\Hs^s(\Gm)$ is denoted by $\Hs^{-s}(\Gm)$, where we make the usual identification $\Ls^2(\Gm) \sst \Hs^{-s}(\Gm)$ via the $\Ls^2(\Gm)$ pairing. The natural duality pairing between $\Hs^{-s}(\Gm)$ and $\Hs^s(\Gm)$ is denoted by $\langle\cdot, \cdot\rangle_{s, \Gm}$. We adopt the definition of the operator $\curlt_\Gm : \Hs^{3/2}(\Gm) \to \HHs^{1/2}_{\times}(\Gm)$ from \cite{BCS2002a}
\[
    \curlt_\Gm \, \gamma(\varphi) = \gm_D(\gradt \, \varphi), \qqqq \fa \varphi \in \Hs^{2}(\Om).
\]
The surface divergence operator $\divt_\Gm : \HHs^{-1/2}_{\times}(\Gm) \to \Hs^{-3/2}(\Gm)$ is then defined as the dual operator to $\curlt_\Gm$ 
\[
    \inprod{\divt_\Gm \ub, \varphi}_{3/2, \Gm} = - \inprod{\ub, \curlt_\Gm \varphi}_{1/2, \Gm}, \qqq \fa \ub \in \HHs^{-1/2}_{\times}(\Gm), \ \fa \varphi \in \Hs^{3/2}(\Gm).
\]
Now, we introduce the space 
\[
    \HHs^{-1/2}_{\times}(\divt_{\Gm}, \Gm) := \brac{\ub \in \HHs^{-1/2}_{\times}(\Gm) : \divt_{\Gm} \ub \in \Hs^{-1/2}(\Gm)},
\]
equipped with the graph norm
\[
    \norm{\ub}^2_{\HHs^{-1/2}_{\times}(\divt_{\Gm}, \Gm)} := \norm{\ub}^2_{\HHs^{-1/2}_{\times}(\Gm)} + \norm{\divt_\Gm \ub}^2_{\Hs^{-1/2}(\Gm)}.
\]
It is noteworthy that $\HHs^{-1/2}_{\times}(\divt_{\Gm}, \Gm)$ is the desired trace space for electromagnetic fields. An important property of the space $\HHs^{-1/2}_{\times}(\divt_{\Gm}, \Gm)$ is its self-duality, which was given in \cite{BC2001b} and \cite[Lemma~5.6]{BCS2002a}.

\begin{theorem}[self-duality of the space $\HHs^{-1/2}_{\times}(\divt_{\Gm}, \Gm)$]
\label{thm:duality}
    The pairing $\langle{\cdot, \cdot}\rangle_{\times, \Gm}$ can be extended to a continuous bilinear form on $\HHs^{-1/2}_{\times}(\divt_{\Gm}, \Gm)$. Moreover, the space $\HHs^{-1/2}_{\times}(\divt_{\Gm}, \Gm)$ becomes its own dual with respect to $\langle{\cdot, \cdot}\rangle_{\times, \Gm}$.
\end{theorem}

Finally, we introduce the trace of the energy space $\HHs(\curlt, \Om)$. The next theorem presents the extension of the tangential trace operator $\gm_D$ to $\HHs(\curlt, \Om)$ and a related integration by parts formula, see \cite[Theorem~4.1]{BCS2002a} or \cite[Theorem~2.3]{CH2012}.

\begin{theorem}[integration by parts formula]
    \label{thm:integration_by_parts}
    The tangential trace operator $\gm_D$ can be extended to a continuous mapping from $\HHs(\curlt, \Om)$ onto $\HHs^{-1/2}_{\times}(\divt_{\Gm}, \Gm)$, which possesses a continuous right inverse. In addition, the following integration by parts formula holds
    \begin{equation}
    \label{eq:identity}
        \int_\Om (\curlt \, \ub \cdot \vb - \ub \cdot \curlt \, \vb) \dx = -\inprod{\gm_D \ub, \gm_D \vb}_{\times, \Gm}, \q \fa \ub, \vb \in \HHs(\curlt, \Om).
    \end{equation}
\end{theorem}
The traces of $\HHs(\curlt^2, \Om)$ involve the Neumann trace operator $\gm_N := \gm_D \circ \curlt$, which continuously maps $\HHs(\curlt^2, \Om)$ onto $\HHs^{-1/2}_{\times}(\divt_{\Gm}, \Gm)$, see \cite{BH2005}.

\section{Potentials and integral operators}
\label{sec:potentials}

In this section, we introduce the relevant potentials and boundary integral operators for electromagnetic scattering. In order to support the numerical analysis in the next sections, the potentials and integral operators are defined for wave number $\sm$, where $\sm = \kappa$ or $\sm = i\kappa$ with $\kappa > 0$. The reader can find more details on the case of a purely imaginary wave number in \cite[Section~5.6.4]{Nedelec2001}, \cite[Section~5.1]{Hiptmair2003b} and \cite{SW2009}. For the sake of convenience, we restate here the wave equation \eqref{eq:wave} and the Silver-M\"{u}ller radiation condition \eqref{eq:SilverMuller}, which are now associated with the wave number $\sm$
\begin{align}
    & \,\, \curlt \, \curlt \, \ub - \sm^2 \ub = \zrb \qqqqqq \text{in} \q \Om \cup \Om^c, \label{eq:wave1} \\
    & \lim_{r \to \infty} \int_{\pa B_r} \abs{\curlt \, \ub \times \nv + i\sm (\nv \times \ub) \times \nv}^2 \ds = 0. \label{eq:SilverMuller1}
\end{align}
    Let $G_{\sm}(\xb, \yb)$ be the fundamental solution associated with the operator $\Delta + \sm^2$, i.e.,
\[
    G_\sm(\xb, \yb) := \dfrac{\exp(i\sm\abs{\xb - \yb})}{4\pi \abs{\xb - \yb}}, \qqqqq \xb \neq \yb.
\]
We use this kernel to define the scalar and vectorial single layer potentials
\[
    \Psi^\sm_V(\varphi)(\xb) := \int_\Gm \varphi(\yb) G_\sm(\xb, \yb) \ds(\yb), \qq \Psib^\sm_{\Ab}(\ub)(\xb) := \int_\Gm \ub(\yb) G_\sm(\xb, \yb) \ds(\yb),
\]
with $\xb \notin \Gm$. For electromagnetic scattering, the following ``Maxwell single and double layer potentials'' are essential
\[
    \Psib^\sm_{SL}(\ub) := \Psib^\sm_{\Ab} (\ub) + \dfrac{1}{\sm^2} \gradt \, \Psi^\sm_V(\divt_\Gm \ub), \qqq \Psib^\sm_{DL}(\ub) := \curlt\, \Psib^\sm_{\Ab}(\ub).
\]
The potentials $\Psib^\sm_{SL}$ and $\Psib^\sm_{DL}$ are continuous mappings from $\HHs^{-1/2}_{\times}(\divt_\Gm, \Gm)$ into $\HHs_{\textup{loc}}(\curlt^2, \Om \cup \Om^c)$, see \cite[Theorem~17]{Hiptmair2003b}. Moreover, for any $\ub \in \HHs^{-1/2}_{\times}(\divt_\Gm, \Gm)$, both $\Psib^{\sm}_{SL}(\ub)$ and $\Psib^{\sm}_{DL}(\ub)$ are solutions to the equation \eqref{eq:wave1} and fulfill the radiation condition \eqref{eq:SilverMuller1}. On the other hand, any solution $\ub \in \HHs_{\textup{loc}}(\curlt^2, \Om \cup \Om^c)$ to \eqref{eq:wave1}-\eqref{eq:SilverMuller1} satisfies the following variant of the well-known Stratton-Chu representation formula (see, e.g., \cite[Theorem~6.2]{Colton1992}, \cite[Theorem~5.5.1]{Nedelec2001} and \cite[Theorem~22]{Hiptmair2003b})
\begin{equation}
\label{eq:representation}
    \ub(\xb) = - \Psib^{\sm}_{DL}(\left[\gm_D\right]_{\Gm} \ub)(\xb) - \Psib^{\sm}_{SL}(\left[\gm_N\right]_{\Gm} \ub) (\xb), \qqq \xb \in \Om \cup \Om^c,
\end{equation}
where $\left[\gm\right]_{\Gm} := \gm^+ - \gm^-$ is the jump of some trace $\gm$ across $\Gm$, and the superscripts $-$ and $+$ designate traces onto $\Gm$ from $\Om$ and $\Om^c$, respectively. 
Next, please bear in mind the fact that
\begin{equation}
\label{eq:sym1}
    \curlt \circ \Psib^\sm_{SL} = \Psib^\sm_{DL}, \qqqqq \curlt \circ \Psib^\sm_{DL} = \sm^2 \Psib^\sm_{SL},
\end{equation}
which immediately implies
\begin{equation}
    \label{eq:sym2}
    \gm^\pm_N \circ \Psib^\sm_{SL} = \gm^\pm_D {\circ} \Psib^\sm_{DL}, \qqqq \gm^\pm_N \circ \Psib^\sm_{DL} = \sm^2 \gm^\pm_D \circ \Psib^\sm_{SL}.
\end{equation}
The ``symmetric'' relations \eqref{eq:sym1} and \eqref{eq:sym2} convince us that only integral operators associated with the Maxwell single layer potential $\Psib^\sm_{SL}$ (or the Maxwell double layer potential $\Psib^\sm_{DL}$) are needed to be explicitly defined. Let $\brac{\gm}_\Gm := \frac{1}{2}(\gm^+ + \gm^-)$ be the average of exterior and interior traces $\gm$ on the boundary $\Gm$. Now, we are ready to define our desired boundary integral operators.
\begin{theorem}
\label{thm:intops}
    For $\sm = \kappa$ or $\sm = i\kappa$ with $\kappa > 0$, the boundary integral operators
    \begin{align*}
        & V_\sm := \brac{\gm}_\Gm \circ \Psi^\sm_V && : \Hs^{-1/2}(\Gm) \to \Hs^{1/2}(\Gm), \\
        & A_\sm := \brac{\gm_D}_\Gm \circ \Psib^\sm_{\Ab} && : \HHs^{-1/2}_\times(\Gm) \to  \HHs^{1/2}_\times(\Gm), \\
        & S_\sm := \brac{\gm_D}_\Gm \circ \Psib^\sm_{SL} && : \HHs^{-1/2}_{\times}(\divt_\Gm, \Gm) \to \HHs^{-1/2}_{\times}(\divt_\Gm, \Gm), \\
        & C_\sm := \brac{\gm_N}_\Gm \circ \Psib^\sm_{SL} && : \HHs^{-1/2}_{\times}(\divt_\Gm, \Gm) \to \HHs^{-1/2}_{\times}(\divt_\Gm, \Gm),
    \end{align*}
    are well defined and continuous. Moreover, the following relation holds
    \begin{equation}
    \label{eq:efio}
        S_\sm = A_\sm + \dfrac{1}{\sm^2} \curlt_\Gm \circ V_\sm \circ \divt_\Gm.
    \end{equation}
\end{theorem}

\begin{proof}
    The definition of integral operators can be found in \cite[Theorem~19]{Hiptmair2003b}. The last assertion follows from the definition of the Maxwell single layer potential $\Psib^\sm_{SL}$ and the fact that $\curlt_\Gm \circ \gm = \gm_D \circ \gradt$.\footnote{In fact, the proof of Theorem~\ref{thm:intops} cannot be obtained solely from arguments stated in this paper. However, to keep the paper concise, we just recall some properties of the potentials and operators that are necessary for further analysis and refer the reader to \cite{BHV+2003,Hiptmair2003b} for more detailed information.}
\end{proof}

In order to infer the exterior and interior traces of the potentials, the following jump relations are crucial \cite[Theorem~18]{Hiptmair2003b}
\begin{equation}
    \label{eq:jump}
    \left[\gm\right]_\Gm \circ \Psi^\sm_V = 0, \qqq \left[\gm_D\right]_\Gm \circ \Psib^\sm_{\Ab} = \zrb, \qqq \left[\gm_N\right]_\Gm \circ \Psib^\sm_{SL} = -Id,
\end{equation}
where $Id$ stands for the identity operator. Therefore, we have the Neumann traces of the Maxwell single layer potential as follows
\[
    \gm^+_N \circ \Psib^\sm_{SL} = -\dfrac{1}{2} Id + C_\sm, \qqqqq \gm^-_N \circ \Psib^\sm_{SL} = \dfrac{1}{2} Id + C_\sm.
\]
The following lemma provides an auxiliary result for establishing the coercivity of the CFIE operator in the next section.

\begin{lemma}
\label{lem:compact1}
    For $\kappa > 0$, the following integral operators are compact
    \begin{align*}
        & \delta V_\kappa := V_\kappa - V_{i\kappa} && : \Hs^{-1/2}(\Gm) \to \Hs^{1/2}(\Gm), \\
        & \delta A_\kappa := A_\kappa - A_{i\kappa} && : \HHs^{-1/2}_{\times}(\Gm) \to \HHs^{1/2}_{\times}(\Gm), \\
        & \delta C_\kappa := C_\kappa - C_{i\kappa} && : \HHs^{-1/2}_{\times}(\divt_\Gm, \Gm) \to \HHs^{-1/2}_{\times}(\divt_\Gm, \Gm), \\
        & \delta S_\kappa := \delta A_\kappa + \kappa^{-2} \curlt_\Gm \circ \delta V_\kappa \circ \divt_\Gm && : \HHs^{-1/2}_{\times}(\divt_\Gm, \Gm) \to \HHs^{-1/2}_{\times}(\divt_\Gm, \Gm).
    \end{align*}
\end{lemma}
\begin{proof}
    The first three assertions can be proven by following the lines of the proof of \cite[Theorem~3.4.1]{Nedelec2001} or \cite[Theorem~3.12]{BHV+2003}. Here, it is noteworthy that the kernel $G_\kappa - G_{i\kappa}$ is regular. The last assertion is an immediate consequence of the first two ones.
\end{proof}

To end this section, we present here some useful properties of integral operators of purely imaginary wave number. These results are taken from \cite[Theorem~21]{Hiptmair2003b} and \cite[Theorem~2.5]{SW2009}.

\begin{theorem}
\label{thm:ellipticity}
    For $\kappa > 0$, the operators $V_{i\kappa}$ and $A_{i\kappa}$ are symmetric with respect to the bilinear pairings $\langle{\cdot, \cdot}\rangle_{1/2, \Gm}$ and $\langle{\cdot, \cdot}\rangle_{\times, \Gm}$, respectively, i.e.,
    \begin{align*}
        & \inprod{\psi, V_{i\kappa} \varphi}_{1/2, \Gm} = \inprod{\varphi, V_{i\kappa} \psi}_{1/2, \Gm}, && \fa \psi, \varphi \in \Hs^{-1/2}(\Gm), \\
        & \inprod{\vb, A_{i\kappa} \ub}_{\times, \Gm} = \inprod{\ub, A_{i\kappa} \vb}_{\times, \Gm}, && \fa \vb, \ub \in \HHs^{-1/2}_\times(\Gm).
    \end{align*}
    In addition, there exists a positive constant $C$ depending only on $\Gm$ such that
    \begin{align*}
        \inprod{\ovl{\varphi}, V_{i\kappa} \varphi}_{1/2, \Gm} & \ge C \norm{\varphi}^2_{\Hs^{-1/2}(\Gm)}, && \fa \varphi \in \Hs^{-1/2}(\Gm), \\
        \inprod{\ovl{\ub}, A_{i\kappa} \ub}_{\times, \Gm} & \ge C \norm{\ub}^2_{\HHs_\times^{-1/2}(\Gm)}, && \fa \ub \in \HHs_\times^{-1/2}(\divt_\Gm, \Gm).
    \end{align*}
\end{theorem}
The following properties of the operator $S_{i\kappa}$ immediately follow from Theorem~\ref{thm:intops} and Theorem~\ref{thm:ellipticity}. They will play a central role in further analysis.
\begin{corollary}
\label{cor:ellipticity}
    For $\kappa > 0$, the integral operator $S_{i\kappa}$ is  $\HHs_\times^{-1/2}(\divt_\Gm, \Gm)$-elliptic and symmetric with respect to $\langle{\cdot, \cdot}\rangle_{\times, \Gm}$, i.e., there exists a constant $C > 0$ depending only on $\Gm$ such that, for all $\ub, \vb \in \HHs^{-1/2}_\times(\divt_\Gm, \Gm)$
    \[
        \inprod{\vb, S_{i\kappa} \ub}_{\times, \Gm} = \inprod{\ub, S_{i\kappa} \vb}_{\times, \Gm}, \qq \inprod{{\ovl{\ub}}, S_{i\kappa} \ub}_{\times, \Gm} \ge C \norm{\ub}^2_{\HHs^{-1/2}_\times(\divt_\Gm, \Gm)}.
    \]
\end{corollary}
 
\section{The combined field integral equation}
\label{sec:cfie}

In this section, we propose a new CFIE which yields a unique solution to \eqref{eq:wave}-\eqref{eq:SilverMuller} for any wave number $\kappa > 0$. A common strategy when combining the Maxwell single layer potential $\Psib^\kappa_{SL}$ and the Maxwell double layer potential $\Psib^\kappa_{DL}$ is to introduce a compact regularization that targets either $\Psib^\kappa_{SL}$ or $\Psib^\kappa_{DL}$ (see \cite{Kress1986} and \cite{BH2005}). In contrast to this approach, in the following formulation, the potentials $\Psib^\kappa_{SL}$ and $\Psib^\kappa_{DL}$ are respectively complemented by their counterparts $\Psib^{i\kappa}_{SL}$ and $\Psib^{i\kappa}_{DL}$ for purely imaginary wave number. This way of combining the operators has been investigated for direct CFIEs, whose unknowns are the physical electric current on the surface, see, e.g., \cite{CDE+2002,MBC+2020}. In this paper, we investigate an indirect formulation that is based on a potential representation for the scattered electric field $\eb$. In particular, we consider the following ansatz 
\begin{equation}
\label{eq:ansatz}
    \eb = \paren{i\eta \, \Psib^{\kappa}_{SL} \circ \gm_D^- \circ \Psib^{i\kappa}_{SL} + \Psib^{\kappa}_{DL} \circ \gm_D^- \circ \Psib^{i\kappa}_{DL}} (\xib),
\end{equation}
where $\xib \in \HHs^{-1/2}_{\times}(\divt_\Gm, \Gm)$ and 
$\eta \in \R \setminus \{0\}$. Taking the exterior tangential trace $\gm^+_D$ of \eqref{eq:ansatz} results in the CFIE
\begin{equation}
\label{eq:cfie}
    \LL_{\kappa}(\xib) = - \gm_D^+ \eb^{in},
\end{equation}
where the boundary integral operator
\begin{equation}
    \label{eq:cfio}
    \LL_{\kappa} := i\eta \, S_{\kappa} \circ S_{i\kappa} + \paren{-\dfrac{1}{2} Id+ C_{\kappa}} \circ \paren{\dfrac{1}{2} Id+ C_{i\kappa}}.
\end{equation}
From now on, we fix the coupling parameter $\eta \in \R \setminus \{0\}$ and keep it constant. The particular impact of this parameter on the CFIE will be discussed later. As $\HHs^{-1/2}_{\times}(\divt_\Gm, \Gm)$ is its own dual with respect to $\langle \cdot, \cdot \rangle_{\times, \Gm}$, the variational formulation of \eqref{eq:cfie} reads as: find $\xib \in \HHs^{-1/2}_{\times}(\divt_\Gm, \Gm)$ such that, for all $\ub \in \HHs^{-1/2}_{\times}(\divt_\Gm, \Gm)$
\begin{equation}
\label{eq:vf}
    \inprod{\ovl{\ub}, \LL_\kappa(\xib)}_{\times, \Gm} = - \inprod{\ovl{\ub}, \gm_D^+ \eb^{in}}_{\times, \Gm}.
\end{equation}

\subsection{Uniqueness}

The most important property of CFIEs which distinguishes them from the standard boundary integral equations is the uniqueness of a solution for any wave number.
\begin{theorem}[Uniqueness]
\label{thm:uniqueness}
    For any wave number $\kappa > 0$, the integral equation \eqref{eq:cfie} has at most one solution $\xib \in \HHs^{-1/2}_{\times}(\divt_\Gm, \Gm)$.
\end{theorem}

\begin{proof}
    Let $\xib \in \HHs^{-1/2}_{\times}(\divt_\Gm, \Gm)$ be a solution to the homogeneous equation
    \[
        \LL_\kappa (\xib) = \zrb.
    \]
    It is clear that the scattered electric field $\eb$ given by \eqref{eq:ansatz} is a solution to the exterior problem \eqref{eq:wave} and \eqref{eq:SilverMuller} with the homogeneous Dirichlet boundary condition $\eb \times \nv = \zrb$. Referring to Relich's lemma, we can conclude that $\eb = \zrb$ in $\Om^c$. Therefore, the jump relations \eqref{eq:jump} can be invoked to get the interior traces
    \[
        \gm_D^- \eb = \paren{ \gm_D^- \circ \Psib^{i\kappa}_{DL}} (\xib), \qqqqqq \gm_N^- \eb = i \eta \paren{\gm_D^- \circ \Psib^{i\kappa}_{SL}} (\xib).
    \]
    On one hand, the integration by parts formula \eqref{eq:identity} gives us that
    \[
        \inprod{\gm_D^- \ovl{\eb}, \gm_N^- \eb}_{\times, \Gm} = \int_\Om \paren{\kappa^2 \abs{\eb(\xb)}^2 - \abs{\curlt \, \eb(\xb)}^2} \dx \in \R.
    \]
    On the other hand, the relations \eqref{eq:sym2} allow us to deduce that
    \begin{align*}
        \inprod{\gm_D^- \ovl{\eb}, \gm_N^- \eb}_{\times, \Gm} 
        & = i\eta \inprod{\paren{\gm_D^- \circ \Psib^{i\kappa}_{DL}} (\ovl{\xib}), \paren{\gm_D^- \circ \Psib^{i\kappa}_{SL}} (\xib)}_{\times, \Gm} \\
        & = i\eta \int_\Om \paren{\kappa^2 \abs{\Psib^{i\kappa}_{SL}(\xib)(\xb)}^2 + \abs{\curlt \, \Psib^{i\kappa}_{SL}(\xib)(\xb)}^2} \dx \in i\R. 
    \end{align*}
    These arguments together with the continuity of the trace operator $\gm_D^- : \HHs(\curlt, \Om) \to \HHs^{-1/2}_\times(\divt_\Gm, \Gm)$ (cf. Theorem~\ref{thm:integration_by_parts}) imply that
    \[
        0 = \norm{\Psib^{i\kappa}_{SL}(\xib)}_{\HHs(\curlt, \Om)} \ge C \norm{S_{i\kappa}(\xib)}_{\HHs^{-1/2}_{\times}(\divt_\Gm, \Gm)},
    \]
    for some positive constant $C$.
    This means $S_{i\kappa}(\xib) = 0$ in $\HHs^{-1/2}_{\times}(\divt_\Gm, \Gm)$. Finally, we invoke the $\HHs^{-1/2}_{\times}(\divt_\Gm, \Gm)$-ellipticity of $S_{i\kappa}$ in Corollary~\ref{cor:ellipticity} to conclude that $\xib = \zrb$, or equivalently, the CFIE \eqref{eq:cfie} has at most one solution in $\HHs^{-1/2}_{\times}(\divt_\Gm, \Gm)$.
\end{proof}

\subsection{Coercivity}
\label{subsec:coercivity}

It is well-known that if the integral operator $\LL_\kappa$ is a Fredholm operator of index $0$, then the injectivity of $\LL_\kappa$ (i.e., Theorem~\ref{thm:uniqueness}) also implies its surjectivity by a Fredholm alternative argument. Therefore, the next task is to prove that the operator $\LL_\kappa$ satisfies a generalized G{\aa}rding inequality. 

\begin{theorem}[Coercivity]
    \label{thm:coercivity}
    For any wave number $\kappa > 0$, there exist a positive constant $C$, an isomorphism $\Theta_\Gm : \HHs^{-1/2}_{\times}(\divt_\Gm, \Gm) \to \HHs^{-1/2}_{\times}(\divt_\Gm, \Gm)$, and a compact sesquilinear form\footnote{Let $\Hs$  be a normed space over $\C$ and $\Hs^\prime$ its dual space. A sesquilinear form $a : \Hs \times \Hs \to \C$ is said to be compact if its associated operator $A : \Hs \to \Hs^\prime$ is compact, i.e.,
    \[
        a(u, \ovl{v}) = \inprod{Au, v}_{\Hs^\prime \times \Hs}, \qqqq \fa u, v \in \Hs.
    \]
    } $c : \HHs^{-1/2}_{\times}(\divt_\Gm, \Gm) \times \HHs^{-1/2}_{\times}(\divt_\Gm, \Gm) \to \C$ such that
    \begin{equation}
    \label{eq:Garding}
        \abs{\inprod{\Theta_\Gm \ovl{\ub}, \LL_\kappa(\ub)}_{\times, \Gm} + c(\ub, \ub)} \ge C \norm{\ub}^2_{\HHs^{-1/2}_{\times}(\divt_\Gm, \Gm)}, \qq \fa \ub \in \HHs^{-1/2}_{\times}(\divt_\Gm, \Gm).
    \end{equation}
\end{theorem}

\begin{proof}
We rewrite the integral operator $\LL_\kappa$ in the following form
\begin{equation}
\label{eq:split}
    \LL_\kappa = \wh{\LL}_{\kappa} + R_\kappa,
\end{equation}
where $\wh{\LL}_{\kappa} : \HHs^{-1/2}_{\times}(\divt_\Gm, \Gm) \to \HHs^{-1/2}_{\times}(\divt_\Gm, \Gm)$ is given by
\[
    \wh{\LL}_{\kappa} := i\eta \, \widetilde{S}_{i\kappa} \circ S_{i\kappa} + \paren{-\dfrac{1}{2} Id + C_{i\kappa}} \circ \paren{\dfrac{1}{2} Id + C_{i\kappa}},
\] 
with
\[
    S_{i\kappa} = A_{i\kappa} - \kappa^{-2} \curlt_\Gm \circ V_{i\kappa} \circ \divt_\Gm, \qqq \widetilde{S}_{i\kappa} = A_{i\kappa} + \kappa^{-2} \curlt_\Gm \circ V_{i\kappa} \circ \divt_\Gm.
\]
Here, the operator ${ R_\kappa}: \HHs^{-1/2}_{\times}(\divt_\Gm, \Gm) \to \HHs^{-1/2}_{\times}(\divt_\Gm, \Gm)$ reads as
\[
    { R_\kappa} := i\eta \, \delta S_\kappa \circ S_{i\kappa} + \delta C_{\kappa} \circ \paren{\dfrac{1}{2} Id + C_{i\kappa}},
\] 
which, according to Lemma~\ref{lem:compact1}, is a compact operator.
The following Calder\'{o}n projection identity can be deduced from the representation formula \eqref{eq:representation} and \cite[Formula~(35)]{BH2003}
\begin{equation}
\label{eq:Calderon}
    \paren{-\dfrac{1}{2} Id + C_{i\kappa}} \circ \paren{\dfrac{1}{2} Id + C_{i\kappa}} = \kappa^2 S_{i\kappa} \circ S_{i\kappa}.
\end{equation}
Next, let us denote
\begin{equation}
\label{eq:Mik}
    \begin{aligned}
        M_{i\kappa} & := \paren{i\eta \, \widetilde{S}_{i\kappa} +  \kappa^2 S_{i\kappa}} \\
        & \ = \paren{\kappa^2 + i\eta} A_{i\kappa} - \paren{1- i\eta \kappa^{-2}} \curlt_\Gm \circ V_{i\kappa} \circ \divt_\Gm.
    \end{aligned}
\end{equation}
By means of Theorem~\ref{thm:ellipticity}, one can easily see that the operator $M_{i\kappa}$ is $\HHs^{-1/2}_{\times}(\divt_\Gm, \Gm)$-elliptic in the sense
\begin{equation}
\label{eq:ellipticity_Mik}
    \begin{aligned}
        \Real \inprod{\ovl{\ub}, M_{i\kappa} \ub}_{\times, \Gm} & = \kappa^2 \inprod{\ovl{\ub}, A_{i\kappa} \ub}_{\times, \Gm} + \inprod{\divt_\Gm \ovl{\ub}, \paren{V_{i\kappa} \circ \divt_\Gm} \ub}_{\times, \Gm} \\
        & \ge C_1 \norm{\ub}_{\HHs^{-1/2}_{\times}(\divt_\Gm, \Gm)}^2,
    \end{aligned}
\end{equation}
for some constant $C_1 > 0$ independent of $\ub$. Then, $\wh{\LL}_{\kappa}$ can be rewritten as
\[
    \wh{\LL}_{\kappa} = M_{i\kappa} \circ S_{i\kappa},
\]
and hence
\[
    \LL_\kappa = M_{i\kappa} \circ S_{i\kappa} + R_\kappa.
\]
Finally, taking into account that the operator $S_{i\kappa}$ is also $\HHs^{-1/2}_{\times}(\divt_\Gm, \Gm)$-elliptic (cf. Corollary~\ref{cor:ellipticity}), we can conclude 
\begin{align*}
    \abs{\inprod{S_{i\kappa} \ovl{\ub}, \LL_\kappa \ub}_{\times, \Gm} - \inprod{S_{i\kappa} \ovl{\ub}, R_\kappa \ub}_{\times, \Gm}} & = \abs{\inprod{S_{i\kappa} \ovl{\ub}, \paren{M_{i\kappa} \circ S_{i\kappa}} \ub}_{\times, \Gm}} \\
    & \ge C_1 \norm{S_{i\kappa} \ub}^2_{\HHs^{-1/2}_{\times}(\divt_\Gm, \Gm)} \ge C \norm{\ub}^2_{\HHs^{-1/2}_{\times}(\divt_\Gm, \Gm)},
\end{align*}
for all $\ub \in \HHs^{-1/2}_{\times}(\divt_\Gm, \Gm)$ and for some constant $C > 0$ independent of $\ub$. In other words, the operator $\LL_\kappa$ satisfies the generalized G{\aa}rding inequality \eqref{eq:Garding} with the isomorphism $\Theta_\Gm := S_{i\kappa}$ and the compact sesquilinear form $c$ defined by
\[
    c(\ub, \vb) := -\inprod{S_{i\kappa} \ovl{\vb}, R_\kappa \ub}_{\times, \Gm}, \qqqq \fa \ub, \vb \in \HHs^{-1/2}_{\times}(\divt_\Gm, \Gm).
\]
\end{proof}
The injectivity of $\LL_\kappa$ (i.e., Theorem~\ref{thm:uniqueness}) and the generalized G{\aa}rding inequality \eqref{eq:Garding} imply that the variational problem \eqref{eq:vf} has a unique solution $\xib \in \HHs^{-1/2}_{\times}(\divt_\Gm, \Gm)$ (cf. \cite[Proposition~3]{BCS2002}).

\section{Galerkin discretization}
\label{sec:discretization}

This section aims to introduce a Galerkin boundary element discretization scheme for the variational problem \eqref{eq:vf}, i.e.,
\[
    \inprod{\ovl{\ub}, \LL_{\kappa}(\xib)}_{\times, \Gm} = - \inprod{\ovl{\ub}, \gm_D^+ \eb^{in}}_{\times, \Gm},
\]
where $\LL_\kappa$ is given by
\[
    \LL_{\kappa} = M_{i\kappa} \circ S_{i\kappa} + i\eta \, \delta S_\kappa \circ S_{i\kappa} + \delta C_{\kappa} \circ \paren{\dfrac{1}{2} Id + C_{i\kappa}}.
\]
The operator $M_{i\kappa}$ is defined by \eqref{eq:Mik} and satisfies the $\HHs^{-1/2}_\times(\divt_\Gm, \Gm)$-ellipticity \eqref{eq:ellipticity_Mik}. The analysis in this section holds for any wave number $\kappa > 0$. Therefore, we omit its specification.

\subsection{Mixed variational formulation}
Since $\LL_\kappa$ consists of composition operators, a direct Galerkin discretization of \eqref{eq:vf} is not possible. A typical approach is to switch \eqref{eq:vf} to a mixed variational formulation. To that end, we introduce two auxiliary unknowns $\vphib \in \HHs^{-1/2}_\times(\divt_\Gm, \Gm)$ and $\psib \in \HHs^{-1/2}_\times(\divt_\Gm, \Gm)$ as follows
\begin{equation}
    \label{eq:unknowns}
    \vphib := S_{i\kappa} \xib, \qqq \text{and} \qqq \psib := \paren{\dfrac{1}{2} Id + C_{i\kappa}} \xib.
\end{equation}
For the sake of brevity, in what follows, we denote by $\XXs_\Gm$ the space $\paren{\HHs^{-1/2}_\times(\divt_\Gm, \Gm)}^3$ equipped with its Euclidean norm.
Then, the mixed variational formulation of \eqref{eq:cfie} reads as: find $\xb := \paren{\xib, \vphib, \psib} \in \XXs_\Gm$ such that for all $\yb := (\ub, \vb, \wb) \in \XXs_\Gm$
\begin{equation}
\label{eq:mixed_vf}
    \begin{aligned}
        \inprod{\ovl{\ub}, M_{i\kappa} \vphib}_{\times, \Gm} + \inprod{\ovl{\ub}, \delta C_\kappa \psib}_{\times, \Gm} & + i\eta \inprod{\ovl{\ub}, \delta S_\kappa \vphib}_{\times, \Gm} && = -\inprod{\ovl{\ub}, \gm_D^+ \eb^{in}}_{\times, \Gm}, \\
        \inprod{\ovl{\vb}, S_{i\kappa} \xib}_{\times, \Gm} & - \inprod{\ovl{\vb}, \vphib}_{\times, \Gm} && = 0, \\
        \inprod{\ovl{\wb}, \paren{\tfrac{1}{2} Id + C_{i\kappa}} \xib}_{\times, \Gm} & -\inprod{\ovl{\wb}, \psib}_{\times, \Gm} && = 0.
    \end{aligned}
\end{equation}
Let $b = a + t$ be the sesquilinear form associated with \eqref{eq:mixed_vf}, where $a: \XXs_\Gm \times \XXs_\Gm \to \C$ and $t: \XXs_\Gm \times \XXs_\Gm \to \C$ are defined by
\begin{align*}
    a(\xb, \yb) & := \inprod{\ovl{\ub}, M_{i\kappa} \vphib}_{\times, \Gm} + \inprod{\ovl{\vb}, S_{i\kappa} \xib}_{\times, \Gm} + \inprod{\ovl{\wb}, \psib}_{\times, \Gm} \\
    & \qq \, - \inprod{\ovl{\vb}, \vphib}_{\times, \Gm} - \inprod{\ovl{\wb}, \paren{\tfrac{1}{2} Id + C_{i\kappa}} \xib}_{\times, \Gm} , \\
    t(\xb, \yb) & := \inprod{\ovl{\ub}, \delta C_\kappa \psib}_{\times, \Gm} + i\eta \inprod{\ovl{\ub}, \delta S_\kappa \vphib}_{\times, \Gm}.
\end{align*}
The sesquilinear form $t$ is compact and the principal part $a$ satisfies the following $\Theta$-coercivity condition.
\begin{lemma}
\label{lem:coercivity_a}
    There exist a constant $C > 0$ and an isomorphism $\Theta : \XXs_\Gm \to \XXs_\Gm$ such that
    \begin{equation}
    \label{eq:semi_elliptic}
        \abs{a\paren{\xb, \Theta \xb}} \ge C \norm{\xb}^2_{\XXs_\Gm}, \qqqq \fa \xb \in \XXs_\Gm.
    \end{equation}
\end{lemma}
\begin{proof}
    For all $\xb := (\xib, \vphib, \psib) \in \XXs_\Gm$, we define the isomorphism $\Theta : \XXs_\Gm \to \XXs_\Gm$ by
    \begin{equation}
        \label{eq:isomorphism}
        \Theta \xb := \paren{\vphib, \alpha \xib, \beta S_{i\kappa}^{-1} \psib},
    \end{equation}
    with some positive constants $\alpha$ and $\beta$. Since the operator $S_{i\kappa}$ is $\HHs^{-1/2}_\times(\divt_\Gm, \Gm)$-elliptic, its inverse $S_{i\kappa}^{-1} : \HHs^{-1/2}_\times(\divt_\Gm, \Gm) \to \HHs^{-1/2}_\times(\divt_\Gm, \Gm)$ exists and is continuous. The following estimate holds
    \begin{equation}
    \label{eq:a}
        \begin{aligned}
            \abs{a\paren{\xb, \Theta \xb}} & \ge \abs{\inprod{\ovl{\vphib}, M_{i\kappa} \vphib}_{\times, \Gm} + \alpha \inprod{\ovl{\xib}, S_{i\kappa} \xib}_{\times, \Gm} + \beta \inprod{S_{i\kappa}^{-1} \ovl{\psib}, \psib}_{\times, \Gm}} \\
            & \qq - \alpha \abs{\inprod{\ovl{\xib}, \vphib}_{\times, \Gm}} - \beta \abs{\inprod{S_{i\kappa}^{-1} \ovl{\psib}, \paren{\tfrac{1}{2} Id + C_{i\kappa}} \xib}_{\times, \Gm}}.
        \end{aligned}
    \end{equation}
    Due to the $\HHs^{-1/2}_\times(\divt_\Gm, \Gm)$-ellipticity and the continuity of $M_{i\kappa}$ and $S_{i\kappa}$, there exist positive constants $C_1, C_2$ and $C_3$ such that
    \begin{align*}
        \Real \inprod{\ovl{\vphib}, M_{i\kappa} \vphib}_{\times, \Gm} & \ge C_1 \norm{\vphib}^2_{\HHs^{-1/2}_\times(\divt_\Gm, \Gm)}, \\
        \inprod{\ovl{\xib}, S_{i\kappa} \xib}_{\times, \Gm} & \ge C_2 \norm{\xib}^2_{\HHs^{-1/2}_\times(\divt_\Gm, \Gm)}, \\
        \inprod{S_{i\kappa}^{-1} \ovl{\psib}, \psib}_{\times, \Gm} & \ge C_2 \norm{S_{i\kappa}^{-1} \psib}^2_{\HHs^{-1/2}_\times(\divt_\Gm, \Gm)} \ge C_3 \norm{\psib}^2_{\HHs^{-1/2}_\times(\divt_\Gm, \Gm)},
    \end{align*}
    for all $\vphib, \xib, \psib \in \HHs^{-1/2}_\times(\divt_\Gm, \Gm)$. By means of the $\veps$-Young inequality, the last two terms on the right-hand side of \eqref{eq:a} can be bounded as follows
    \begin{align*}
        & \abs{\inprod{\ovl{\xib}, \vphib}_{\times, \Gm}} \le \veps \norm{\xib}^2_{\HHs^{-1/2}_\times(\divt_\Gm, \Gm)} + C_\veps \norm{\vphib}^2_{\HHs^{-1/2}_\times(\divt_\Gm, \Gm)}, \\
        & \abs{\inprod{S_{i\kappa}^{-1} \ovl{\psib}, \paren{\tfrac{1}{2} Id + C_{i\kappa}} \xib}_{\times, \Gm}} \le \veps \norm{\psib}^2_{\HHs^{-1/2}_\times(\divt_\Gm, \Gm)} + C_\veps \norm{\xib}^2_{\HHs^{-1/2}_\times(\divt_\Gm, \Gm)},
    \end{align*}
    for an arbitrarily small constant $\veps > 0$ and some constant $C_\veps > 0$ dependent of $\veps$. Combining all above estimates, we arrive at
    \[
        \abs{a\paren{\xb, \Theta \xb}} \ge \min \brac{C_1 - \alpha C_\veps, \alpha (C_2 - \veps) - \beta C_\veps, \beta (C_3 - \veps)} \norm{\xb}^2_{\XXs_\Gm}.
    \]
    Finally, we can choose sufficiently small positive constants $\veps, \alpha$ and $\beta$ (firstly $\veps$, then $\alpha$, and finally $\beta$) such that
    \[
        \min \brac{C_1 - \alpha C_\veps, \alpha (C_2 - \veps) - \beta C_\veps, \beta (C_3 - \veps)} \ge C,
    \]
    for some constant $C > 0$.
\end{proof}
As a result, the sesquilinear form $b$ satisfies a generalized G{\aa}rding inequality. Invoking a Fredholm alternative argument, the unique solvability of \eqref{eq:mixed_vf} follows from the uniqueness of its solution (obtained from that of \eqref{eq:vf}). In particular, the following continuous inf-sup condition is satisfied 
\begin{equation}
    \sup_{\yb \in \XXs_\Gm \setminus \brac{\zrb}} \dfrac{\abs{b(\xb, \yb)}}{\norm{\yb}_{\XXs_\Gm}} \ge C \norm{\xb}_{\XXs_\Gm} \qqqqq \fa \xb \in \XXs_\Gm.
\end{equation}

\subsection{Galerkin boundary element discretization}
Let $\Om$ be a polyhedron and $(\Gm_h)_{h > 0}$ be a family of shape-regular, quasi-uniform triangulations of the surface $\Gm$, which only comprise flat triangles \cite{Ciarlet2002,Monk2003}. The parameter $h$ stands for the meshwidth, which equals the length of the longest edge of triangulation $\Gm_h$. We denote by $\TT_h$ and $\EE_h$ the sets of all triangles and edges of $\Gm_h$, respectively. On each triangle $T \in \TT_h$, we equip the lowest-order triangular Raviart-Thomas space\footnote{In computational engineering, the Raviart-Thomas boundary elements are known as Rao-Wilton-Glisson (RWG) boundary elements.} \cite{RT1977}
\[
    \RT_0(T) := \brac{\xb \mapsto \ab + b \xb : \ab \in \C^2, b \in \C}.
\]
This local space gives rise to the global $\divt_\Gm$-conforming boundary element space
\[
    \UUs_h := \brac{\ub_h \in \HHs^{-1/2}_\times(\divt_\Gm, \Gm) : \left.\ub_h\right|_T \in \RT_0(T), \, \fa T \in \TT_h},
\]
endowed with the edge degrees of freedom \cite{HS2003b}
\[
    \phi_e(\ub_h) := \int_e (\ub_h \times \nv_j) \cdot \ds, \qqqq \fa e \in \EE_h,
\]
where $\nv_j$ is the normal vector of a triangle $T$ in whose closure $e$ is contained. 

Next, we search for a boundary element space $\VVs_h \sst \HHs^{-1/2}_\times(\divt_\Gm, \Gm)$ such that $\dim \VVs_h = \dim \UUs_h$ and the duality paring $\langle{\cdot, \cdot}\rangle_{\times, \Gm}$ on $\VVs_h \times \UUs_h$ satisfies the uniform discrete inf-sup condition
\begin{equation}
    \label{eq:inf_sup_dual}
    \sup_{\vb_h \in \VVs_h \setminus \brac{\zrb}} \dfrac{\abs{\inprod{\ovl{\vb_h}, \ub_h}_{\times, \Gm}}}{\norm{\vb_h}_{\HHs^{-1/2}_\times(\divt_\Gm, \Gm)}} \ge C \norm{\ub_h}_{\HHs^{-1/2}_\times(\divt_\Gm, \Gm)}, \qq \fa \ub_h \in \UUs_h.
\end{equation}
As suggested in \cite[Section~4]{Hiptmair2006}, a suitable choice for $\VVs_h$ is the space of Buffa-Christiansen $\divt_\Gm$-conforming boundary elements defined on the barycentric refinement of $\Gm_h$ \cite{BC2007}. It is worthy noting that the unions $(\UUs_h)_{h > 0}$ and $(\VVs_h)_{h > 0}$ are dense in $\HHs^{-1/2}_\times(\divt_\Gm, \Gm)$, see \cite[Corollary~5]{BH2003} and \cite{BC2007}.

Now, we are in the position to propose a Galerkin boundary element discretization of the mixed variational problem \eqref{eq:mixed_vf}. Let us denote the discrete product subspaces 
\[
    \XXs_h := \VVs_h \times \UUs_h \times \UUs_h \sst \XXs_\Gm, \qq \text{and} \qq \YYs_h := \UUs_h \times \VVs_h\times \VVs_h \sst \XXs_\Gm.
\]
Obviously, $(\XXs_h)_{h > 0}$ and $(\YYs_h)_{h > 0}$ are dense in $\XXs_\Gm$. Then, the discrete formulation of \eqref{eq:mixed_vf} reads as: find $\xb_h := \paren{\xib_h, \vphib_h, \psib_h} \in \XXs_h$ such that for all $\yb_h := \paren{\ub_h, \vb_h, \wb_h} \in \YYs_h$
\begin{equation}
\label{eq:dis_vf}
    \begin{aligned}
        \inprod{\ovl{\ub_h}, M_{i\kappa} \vphib_h}_{\times, \Gm} + \inprod{\ovl{\ub_h}, \delta C_\kappa \psib_h}_{\times, \Gm} & + i\eta \inprod{\ovl{\ub_h}, \delta S_\kappa \vphib_h}_{\times, \Gm} && = -\inprod{\ovl{\ub_h}, \gm_D^+ \eb^{in}}_{\times, \Gm}, \\
        \inprod{\ovl{\vb_h}, S_{i\kappa} \xib_h}_{\times, \Gm} & - \inprod{\ovl{\vb_h}, \vphib_h}_{\times, \Gm} && = 0, \\
        \inprod{\ovl{\wb_h}, \paren{\tfrac{1}{2} Id + C_{i\kappa}} \xib_h}_{\times, \Gm} & - \inprod{\ovl{\wb_h}, \psib_h}_{\times, \Gm} && = 0.
    \end{aligned}
\end{equation}
The following result is analogous to Lemma~\ref{lem:coercivity_a}.
\begin{lemma}
\label{lem:coercivity_a_h}
    There exist a constant $C > 0$ independent of $h$ and an isomorphism $\Theta_h : \XXs_h \to \YYs_h$ such that
    \begin{equation}
    \label{eq:semi_elliptic_h}
        \abs{a\paren{\xb_h, \Theta_h \xb_h}} \ge C \norm{\xb_h}^2_{\XXs_\Gm}, \qqqq \fa \xb_h \in \XXs_h.
    \end{equation}
\end{lemma}
\begin{proof}
    Let us firstly define the operators $I_h : \UUs_h \to \VVs_h^\prime$ and $S_h : \VVs_h \to \VVs_h^\prime$ by
    \begin{equation}
        \label{eq:discrete_operator}
        \begin{aligned}
            \inprod{I_h \ub_h, \vb_h}_{\VVs_h^\prime \times \VVs_h} & := \inprod{\ovl{\vb_h}, \ub_h}_{\times, \Gm}, && \fa \ub_h \in \UUs_h, \, \vb_h \in \VVs_h, \\
            \inprod{S_h \wb_h, \vb_h}_{\VVs_h^\prime \times \VVs_h} & := \inprod{\ovl{\vb_h}, S_{i\kappa} \wb_h}_{\times, \Gm}, && \fa \vb_h, \wb_h \in \VVs_h.
        \end{aligned}
    \end{equation}
    Since the operator $S_{i\kappa}$ is $\HHs^{-1/2}_\times(\divt_\Gm, \Gm)$-elliptic, $S_h$ is invertible following the Lax-Milgram theorem. The discrete inf-sup condition \eqref{eq:inf_sup_dual} implies that the operator $I_h$ is injective. Hence, $I_h$ is bijective as $\dim \VVs_h = \dim \UUs_h$. Thus, $S_h^{-1} \circ I_h$ is an isomorphism between $\UUs_h$ and $\VVs_h$. In addition, the following identity holds 
    \begin{equation}
        \label{eq:identity_h}
        \inprod{\ovl{\vb_h}, (S_{i\kappa} \circ S_h^{-1} \circ I_h) \ub_h}_{\times, \Gm} = \inprod{\ovl{\vb_h}, \ub_h}_{\times, \Gm}, \qqq \fa \ub_h \in \UUs_h, \, \vb_h \in \VVs_h.
    \end{equation}
    On one hand, choosing $\vb_h = (S_h^{-1} \circ I_h) \ub_h$ in \eqref{eq:identity_h} leads to the following estimate
    \begin{equation}
    \label{eq:uniform_boundedness}
        \norm{(S_h^{-1} \circ I_h) \ub_h}_{\HHs^{-1/2}_\times(\divt_\Gm, \Gm)} \le C \norm{\ub_h}_{\HHs^{-1/2}_\times(\divt_\Gm, \Gm)}, \qq \fa \ub_h \in \UUs_h.
    \end{equation}
    On the other hand, the identity \eqref{eq:identity_h} and the uniform discrete inf-sup condition \eqref{eq:inf_sup_dual} allow us to deduce that
    \begin{align}
        C_1 \norm{(S_h^{-1} \circ I_h) \ub_h}_{\HHs^{-1/2}_\times(\divt_\Gm, \Gm)} 
        & \ge \sup_{\vb_h \in \VVs_h \setminus \brac{\zrb}} \dfrac{\abs{\inprod{\ovl{\vb_h}, (S_{i\kappa} \circ S_h^{-1} \circ I_h) \ub_h}_{\times, \Gm}}}{\norm{\vb_h}_{\HHs^{-1/2}_\times(\divt_\Gm, \Gm)}} \nonumber \\
        & = \sup_{\vb_h \in \VVs_h \setminus \brac{\zrb}} \dfrac{\abs {\inprod{\ovl{\vb_h}, \ub_h}_{\times, \Gm}}}{\norm{\vb_h}_{\HHs^{-1/2}_\times(\divt_\Gm, \Gm)}} \label{eq:inverse_S} \\
        & \ge C_2 \norm{\ub_h}_{\HHs^{-1/2}_\times(\divt_\Gm, \Gm)}, \nonumber
    \end{align}
    for all $\ub_h \in \UUs_h$. Now, we can define the isomorphism $\Theta_h : \XXs_h \to \YYs_h$ by
    \begin{equation}
        \label{eq:isomorphism_h}
        \Theta_h \xb_h := \paren{\vphib_h, \alpha \xib_h, \beta (S_h^{-1} \circ I_h) \psib_h}, \qq \fa \xb_h = \paren{\xib_h, \vphib_h, \psib_h} \in \XXs_h,
    \end{equation}
    with some positive $h$-independent constants $\alpha$ and $\beta$ . Following the lines of the proof of Lemma~\ref{lem:coercivity_a}, there exists a constant $C > 0$ independent of $h$ such that
    \begin{equation}
    \label{eq:estimate_a_h}
            \abs{a\paren{\xb_h, \Theta_h \xb_h}} 
            \ge C \norm{\xb_h}^2_{\XXs_\Gm}, \qqqq \fa \xb_h \in \XXs_h.
    \end{equation}
    Here, the uniform boundedness \eqref{eq:uniform_boundedness} and the following estimate were used
    \begin{align*}
        \inprod{(S_h^{-1} \circ I_h) \ovl{\psib_h}, \psib_h}_{\times, \Gm} 
        & = \inprod{(S_h^{-1} \circ I_h) \ovl{\psib_h}, (S_{i\kappa} \circ S_h^{-1} \circ I_h) \psib_h}_{\times, \Gm} \\
        & \ge C_1 \norm{(S_h^{-1} \circ I_h) \psib_h}^2_{\HHs^{-1/2}_\times(\divt_\Gm, \Gm)} \\
        & \ge C_2 \norm{\psib_h}^2_{\HHs^{-1/2}_\times(\divt_\Gm, \Gm)}.
    \end{align*}
    We note that $\alpha$ and $\beta$ in \eqref{eq:isomorphism_h} are not necessarily the same as in \eqref{eq:isomorphism}, and the positive $h$-independent constants $C, C_1$ and $C_2$ can be different in different contexts.
\end{proof}
The following discrete inf-sup condition is crucial for proving the unique solvability of the discrete variational problem \eqref{eq:dis_vf}. This result hints at a generalization of \cite[Theorem~4.2.9]{Sauter2011} for compact perturbations of a $T$-coercive operator. An equivalent result for compactly perturbed bijective operators is stated in \cite[Exercise~26.5]{Ern2021b}.
\begin{lemma}
\label{lem:dis_inf_sup}
    There exist an $h_0 > 0$ and a constant $C > 0$ such that for all $h < h_0$
    \begin{align}
        & \sup_{\yb_h \in \YYs_h \setminus \brac{\zrb}} \dfrac{\abs{b(\xb_h, \yb_h)}}{\norm{\yb_h}_{\XXs_\Gm}} \ge C \norm{\xb_h}_{\XXs_\Gm}, && \fa \xb_h \in \XXs_h. \label{eq:dis_inf_sup}
    \end{align}
\end{lemma}
\begin{proof}
     Our proof mainly follows the approach in the proof of \cite[Theorem~4.2.9]{Sauter2011}. Let us define the operators $B : \XXs_\Gm \to \XXs_\Gm^\prime$ and $B_h : \XXs_h \to \YYs_h^\prime$ by
    \begin{align*}
        \inprod{B\xb, \yb}_{\XXs_\Gm^\prime \times \XXs_\Gm} & := b(\xb, \yb), && \fa \xb, \yb \in \XXs_\Gm, \\
        \inprod{B_h\xb_h, \yb_h}_{\YYs_h^\prime \times \YYs_h} & := b(\xb_h, \yb_h), && \fa \xb_h \in \XXs_h, \yb_h \in \YYs_h.
    \end{align*}
    We prove the discrete inf-sup condition \eqref{eq:dis_inf_sup} by contradiction. For this purpose, we assume that there exists a sequence $\{\wt{\xb}_h\}_h$ with $\wt{\xb}_h := (\wt{\xib}_h, \wt{\vphib}_h, \wt{\psib}_h) \in \XXs_h$ and $\norm{\wt{\xb}_h}_{\XXs_\Gm} = 1$ such that 
    \begin{equation}
    \label{eq:contradiction}
        \norm{B_h\wt{\xb}_h}_{\YYs_h^\prime} \ra 0, \qqq \text{for } \, h \ra 0.
    \end{equation}
    Since $\brac{\wt{\xb}_h}_h$ is bounded in $\XXs_\Gm$, there exists a subsequence (also denoted by $\brac{\wt{\xb}_h}_h$) such that
    \[
        \wt{\xb}_h \sra \wt{\xb} \qq \text{in } \, \XXs_\Gm, \qqq \text{for } \, h \ra 0.
    \]
    Following the lines of the proof of \cite[Theorem~4.2.9]{Sauter2011}, we have that $B\wt{\xb} = 0$ in $\XXs_\Gm^\prime$. This result together with the injectivity of $B$ (i.e., the uniqueness of a solution to \eqref{eq:mixed_vf}) implies that $\wt{\xb} = \zrb$.
    
    Next, we show the strong convergence $\wt{\xb}_h \to \wt{\xb}$ in $\XXs_\Gm$. According to Lemma~\ref{lem:coercivity_a_h}, there exists a constant $C > 0$ such that
    \begin{equation}
    \label{eq:estimate_a}
            \abs{a\paren{\wt{\xb}_h, \Theta_h \wt{\xb}_h}} 
            \ge C \norm{\wt{\xb}_h}^2_{\XXs_\Gm} = C \norm{\wt{\xb} - \wt{\xb}_h}^2_{\XXs_\Gm},
    \end{equation}
    with $\Theta_h$ defined in \eqref{eq:isomorphism_h}. The last equality holds due to $\wt{\xb} = \zrb$. As the sesquilinear form $t$ is compact, there exists a subsequence of $\brac{\wt{\xb}_h}_h$ (also denoted by $\brac{\wt{\xb}_h}_h$) such that
    \[
        \sup_{\yb \in \XXs_\Gm \setminus \brac{\zrb}, \, \norm{\yb}_{\XXs_\Gm} = 1} \abs{t(\wt{\xb}_h, \yb) - t(\wt{\xb}, \yb)} =: \delta_h \to 0 \qqq \text{for } \, h \to 0.
    \]
    Then, by means of the uniform boundedness \eqref{eq:uniform_boundedness} and the fact $\wt{\xb} = \zrb$, it holds that
    \begin{equation}
    \label{eq:estimate_t}
        \abs{t(\wt{\xb}_h, \Theta_h \wt{\xb}_h)} = \abs{t(\wt{\xb}_h, \Theta_h  \wt{\xb}_h) - t(\wt{\xb}, \Theta_h \wt{\xb}_h)} \le \delta_h \norm{\Theta_h \wt{\xb}_h}_{\XXs_\Gm} \le C_1 \delta_h =: \tilde{\delta}_h.
    \end{equation}
    Now, combining \eqref{eq:estimate_a} and \eqref{eq:estimate_t} together leads us to the following estimate
    \[
       \abs{b(\wt{\xb}_h, \Theta_h \wt{\xb}_h)} = \abs{a(\wt{\xb}_h, \Theta_h \wt{\xb}_h) + t(\wt{\xb}_h, \Theta_h \wt{\xb}_h)} \ge C \norm{\wt{\xb} - \wt{\xb}_h}^2_{\XXs_\Gm} - \tilde{\delta}_h.
    \]
    Finally, taking the assumption \eqref{eq:contradiction} into account, we end up with
    \[
        C \norm{\wt{\xb} - \wt{\xb}_h}^2_{\XXs_\Gm} \le \abs{b(\wt{\xb}_h, \Theta_h \wt{\xb}_h)} + \tilde{\delta}_h \to 0 \qqq \text{for } \, h \to 0.
    \]
    It means that $\wt{\xb}_h \to \wt{\xb} = \zrb$ when $h \to 0$, which is a contradiction to the assumption $\norm{\wt{\xb}_h}_{\XXs_\Gm} = 1$. Therefore, the discrete inf-sup condition \eqref{eq:dis_inf_sup} is satisfied.
\end{proof}
The unique solvability of the discrete variational problem \eqref{eq:dis_vf} and the convergence of Galerkin solutions are a direct consequence of Lemma~\ref{lem:dis_inf_sup}, see \cite[Theorem~4.2.1]{Sauter2011}. 
\begin{theorem}
\label{thm:dis_solvability}
    There exists an $h_0 > 0$ such that for all $h < h_0$, the discrete problem \eqref{eq:dis_vf} has a unique solution $\xb_h \in \XXs_h$. In addition, the discrete solutions $\xb_h$ converge to the solution $\xb \in \XXs_\Gm$ of the problem \eqref{eq:mixed_vf} and satisfy the quasi-optimal error estimate
    \begin{equation}
    \label{eq:quasioptimal_error}
        \norm{\xb - \xb_h}_{\XXs_\Gm} \le C \min_{\wt{\xb}_h \in \XXs_h} \norm{\xb - \wt{\xb}_h}_{\XXs_\Gm},
    \end{equation}
    for some constant $C > 0$ independent of $h$.
\end{theorem}

\subsection{Matrix representation}
Now, we examine the matrix representation of the Galerkin discrete system \eqref{eq:dis_vf} and its conditioning property. Let $\{\ub_1, \ub_2, \ldots, \ub_N\}$ and $\{\vb_1, \vb_2, \ldots, \vb_N\}$ be bases of $\UUs_h$ and $\VVs_h$, respectively, with $N := \dim \UUs_h = \dim \VVs_h$. We introduce the following Galerkin matrices 
\begin{align*}
    & {[\GGs]}_{mn} := \inprod{\ovl{\vb_m}, \ub_n}_{\times, \Gm}, 
    && {[\MMs]}_{mn} := \inprod{\ovl{\ub_m}, M_{i\kappa} \ub_n}_{\times, \Gm}, \\
    & {[\SSs]}_{mn} := \inprod{\ovl{\vb_m}, S_{i\kappa} \vb_n}_{\times, \Gm}, 
    && {[\ZZs]}_{mn} := i \eta \inprod{\ovl{\ub_m}, \delta S_{\kappa} \ub_n}_{\times, \Gm}, 
     \\
    & {[\CCs]}_{mn} := \inprod{\ovl{\ub_m}, \delta C_{\kappa} \ub_n}_{\times, \Gm}, 
    && {[\KKs]}_{mn} := \inprod{\ovl{\vb_m}, \paren{\tfrac{1}{2} Id + C_{i\kappa}} \vb_n}_{\times, \Gm},
\end{align*}
with $m, n = 1, 2, \ldots, N$. A discrete solution $\xb_h \in \XXs_h$ to \eqref{eq:dis_vf} can be represented as
\begin{equation}
\label{eq:discrete_solution}
    \xb_h = (\xib_h, \vphib_h, \psib_h) = \sum_{m = 1}^N (\wh{\xi}_m \vb_m, \wh{\vphi}_m \ub_m, \wh{\psi}_m \ub_m).
\end{equation}
Then, the expansion coefficient vector $[\wh{\xb}_h]_m := [(\wh{\xib}_h, \wh{\vphib}_h, \wh{\psib}_h)]_m := (\wh{\xi}_m, \wh{\vphi}_m,\wh{\psi}_m)$ is the solution to the following block matrix system 
\begin{equation}
\label{eq:matrix_system}
    \begin{pmatrix}
        \zrb & \MMs + \ZZs & \CCs \\
        \SSs & - \GGs & \zrb \\
        \KKs & \zrb & - \GGs 
    \end{pmatrix}
    \begin{pmatrix}
        \wh{\xib}_h \\
        \wh{\vphib}_h \\
        \wh{\psib}_h
    \end{pmatrix}
    =
    \begin{pmatrix}
        \bb \\
        \zrb \\
        \zrb
    \end{pmatrix},
\end{equation}
where the right-hand side vector ${[\bb]}_m := - \inprod{\ovl{\ub_m}, \gm^+_D \eb^{in}}_{\times, \Gm}$. Since the matrix $\GGs$ is invertible (cf. the discrete inf-sup condition \eqref{eq:inf_sup_dual}), the Schur complement of $\GGs$ in \eqref{eq:matrix_system} reads as
\begin{equation}
\label{eq:discrete_cfie}
    \LLs \wh{\xib}_h := \paren{\MMs \GGs^{-1} \SSs + \ZZs \GGs^{-1} \SSs + \CCs \GGs^{-1} \KKs} \wh{\xib}_h = \bb.
\end{equation}
This matrix equation represents a boundary element discretization of the variational formulation \eqref{eq:vf}. The following result is an immediate consequence of Theorem~\ref{thm:dis_solvability}.
\begin{corollary}
    There exists an $h_0 > 0$ such that for all $h < h_0$, the matrix equation \eqref{eq:discrete_cfie} has a unique solution $\wh{\xib}_h \in \C^N$. In addition, the corresponding $\xib_h \in \VVs_h$ defined by \eqref{eq:discrete_solution} is a quasi-optimal approximation to the continuous solution $\xib$ of \eqref{eq:vf} in the sense of \eqref{eq:quasioptimal_error}.
\end{corollary}

Next, let $\lambda_{\max}(\AAs)$ and $\lambda_{\min}(\AAs)$ be the maximal and minimal (by moduli) eigenvalues of a square matrix $\AAs$, respectively. The spectral condition number of $\AAs$ is defined by
\[
    \cond(\AAs) := \dfrac{\abs{\lambda_{\max}(\AAs)}}{\abs{\lambda_{\min}(\AAs)}}.
\]
We end up with the following conditioning property of the matrix system \eqref{eq:discrete_cfie}. 
\begin{theorem}
    \label{thm:condition_number}
    {There exist an $h_0 > 0$ and a constant $C > 0$} such that for all meshwidth $h < h_0$
    \[
        \cond(\GGs^{-\transpose} \LLs) \le C.
    \]  
\end{theorem}
\begin{proof}
    We can rewrite the matrix $\LLs$ in the following form
    \[
        \LLs = \paren{\MMs + \ZZs + \CCs \GGs^{-1} \KKs \SSs^{-1} \GGs} \GGs^{-1} \SSs =: \DDs \GGs^{-1} \SSs.
    \]
    The invertibility of $\LLs$ implies that $\DDs$ is also invertible. Let $l_h : \VVs_h \times \UUs_h \to \C$ and $d_h : \UUs_h \times \UUs_h \to \C$ be the sesquilinear forms that produce the matrices $\LLs$ and $\DDs$, respectively. Obviously, $d_h$ is uniformly bounded. The discrete inf-sup condition \eqref{eq:dis_inf_sup} implies the uniform discrete inf-sup condition for $l_h$, which allows us to deduce that
    \begin{align*}
        \sup_{\wb_h \in \UUs_h \setminus \{\zrb\}} \dfrac{\abs{d_h(\ub_h, \wb_h)}}{\norm{\wb_h}_{\HHs^{-1/2}_\times(\divt_\Gm, \Gm)}} 
        & = \sup_{\wb_h \in \UUs_h \setminus \{\zrb\}} \dfrac{\abs{l_h\paren{(S_h^{-1} \circ I_h) \ub_h, \wb_h}}}{\norm{ \wb_h}_{\HHs^{-1/2}_\times(\divt_\Gm, \Gm)}} \\
        & \ge C_1 \norm{(S_h^{-1} \circ I_h) \ub_h}_{\HHs^{-1/2}_\times(\divt_\Gm, \Gm)} \\
        & \ge C_2 \norm{\ub_h}_{\HHs^{-1/2}_\times(\divt_\Gm, \Gm)},
    \end{align*}
    for all $\ub_h \in \UUs_h$ and for some constants $C_1, C_2 > 0$ independent of $h$. The operators $I_h : \UUs_h \to \VVs_h^\prime$ and $S_h : \VVs_h \to \VVs_h^\prime$ are defined by \eqref{eq:discrete_operator}. 
    The last inequality holds due to the inequality \eqref{eq:inverse_S}. According to \cite{CN2002} and \cite[Theorem~2.1]{Hiptmair2006}, the spectral condition number of the matrix $\GGs^{-\transpose} \DDs \GGs^{-1} \SSs$ is uniformly bounded.
\end{proof}

\begin{remark} 
    In computational engineering, the direct counterpart of the CFIE \eqref{eq:cfie} is more commonly used. It can be derived from the Stratton-Chu representation formula \eqref{eq:representation}. More specifically, taking the exterior tangential and Neumann traces of \eqref{eq:representation} for the solution $\eb$ of \eqref{eq:wave}-\eqref{eq:SilverMuller} respectively gives the standard EFIE and MFIE 
    \[
        S_\kappa (\gm_N^+ \eb) = \paren{\dfrac{1}{2} Id + C_\kappa} (\gm_D^+ \eb^{in}), \qqq \paren{\dfrac{1}{2} Id + C_\kappa} (\gm_N^+ \eb) = \kappa^2 S_\kappa (\gm_D^+ \eb^{in}).
    \]
    Left multiplying the EFIE and MFIE by $i\eta S_{i\kappa}$ and $-\tfrac{1}{2} Id + C_{i\kappa}$, respectively, then combining the resulting equations together leads us to the following CFIE
    \begin{equation}
        \label{eq:direct_CFIE}
        \begin{aligned}
            \wt{\LL}_\kappa (\gm_N^+ \eb) 
            & := \paren{i\eta \, S_{i\kappa} \circ S_{\kappa} + \paren{-\dfrac{1}{2} Id + C_{i\kappa}} \circ \paren{\dfrac{1}{2} Id + C_{\kappa}}} (\gm_N^+ \eb) \\
            & \,\, = \paren{i\eta \, S_{i\kappa} \circ \paren{\dfrac{1}{2} Id + C_\kappa} + \kappa^2 \paren{-\dfrac{1}{2} Id + C_{i\kappa}} \circ S_\kappa} (\gm_D^+ \eb^{in}).
        \end{aligned}
    \end{equation}
    The direct formulation \eqref{eq:direct_CFIE} is similar to the Yukawa-Calder\'{o}n CFIE introduced in \cite{MBC+2020}. The injectivity and coercivity of $\wt{\LL}_\kappa$ can be obtained by following the same approaches in the proof of Theorems~\ref{thm:uniqueness} and \ref{thm:coercivity}. The injectivity of $\wt{\LL}_\kappa$ was also shown in \cite{MBC+2020}, which was based on a standard approach from \cite{BEP+2009}.

    A typical Galerkin boundary element discretization of the operator $\wt{\LL}_\kappa$ reads as (see, e.g., \cite{MBC+2020,LCA+2023})
    \[
        \wt{\LLs} := \SSs \GGs^{-\transpose} \wt{\SSs} + \NNs \GGs^{-1} \wt{\NNs},
    \]
    where the Galerkin matrices are given by
    \begin{align*}
        & [\SSs]_{mn} := \inprod{\ovl{\vb_m}, S_{i\kappa} \vb_n}_{\times, \Gm}, \
        && [\wt{\SSs}]_{mn} := i\eta \inprod{\ovl{\ub_m}, S_\kappa \ub_n}_{\times, \Gm}, \\
        & [\NNs]_{mn} := \inprod{\ovl{\vb_m}, (-\tfrac{1}{2} Id + C_{i\kappa}) \, \ub_n}_{\times, \Gm}, 
        && [\wt{\NNs}]_{mn} := \inprod{\ovl{\vb_m}, (\tfrac{1}{2} Id + C_{\kappa}) \, \ub_n}_{\times, \Gm}.
    \end{align*}
    Several numerical experiments in \cite{MBC+2020,LCA+2023} showed that the corresponding discrete problem is uniquely solvable for all wave numbers $\kappa > 0$, and the matrix $\wt{\LLs}$ is well-conditioned both when a fine mesh is used and when a low frequency is considered. A similar discretization scheme can also be applied for the variational problem \eqref{eq:vf}. However, proofs of the corresponding discrete inf-sup conditions and the uniform boundedness of condition number are still missing from literature. We note that those proofs do not directly benefit from the arguments established in Lemma~\ref{lem:dis_inf_sup} and Theorem~\ref{thm:condition_number}. Therefore, they require a separate examination, which falls beyond the scope of this paper.
\end{remark}

\begin{remark}
    \label{rem:ikappa}
    The CFIE formulation \eqref{eq:cfie} can be generalized when replacing the purely imaginary wave number $i\kappa$ by any wave number $i\kappa^\prime$ with $\kappa^\prime > 0$. The analysis can immediately be extended to that case.
\end{remark}

\section{Numerical results}
\label{sec:results}

In this section, we perform some numerical experiments to evaluate the performance of the proposed CFIE and its Galerkin discretization. The scattering of electromagnetic waves by a sphere and a unit cube are considered. In all experiments, excitation is provided by an incident plane wave with electric field
\[
    \eb^{in}(\xb) = \hat{\xb} \exp(i\kappa \, \hat{\zb} \cdot \xb),
\]
where $\hat{\xb}$ and $\hat{\zb}$ stand for the unit vectors along the $x$-axis and $z$-axis, respectively. 

\subsection{Sphere}

We start with the scattering by a perfectly conducting sphere of radius $1\mathrm{m}$, centered at the origin of the coordinate system. 

The first experiment aims at demonstrating the solvability of the proposed CFIE \eqref{eq:cfie} {and the convergence of its Galerkin discretization \eqref{eq:discrete_cfie}}. We choose $\kappa = 4.4934$ which is a resonant wave number associated with the unit sphere (see \cite[Section~7.2]{Jin2010}). The exact solution $\eb_{\mathrm{exct}}$ to the problem \eqref{eq:wave}-\eqref{eq:SilverMuller} can be derived via the famous Mie series, see \cite[Section~9.5]{Monk2003} and \cite[Section~7.4.3]{Jin2010}. 

The unit sphere is approximated by planar triangulations with meshwidth $h$. The discrete solutions $\xib_h$ are computed by solving the matrix equation \eqref{eq:discrete_cfie} using the GMRES iterative method with tolerance $\veps_0 = 10^{-12}$. The corresponding scattered fields $\eb_h$ are computed based on the potential ansatz \eqref{eq:ansatz}. The coupling parameter $\eta = - \kappa^2$ is chosen, and four purely imaginary wave numbers $i\kappa^\prime$ (instead of $i\kappa)$, with $\kappa^\prime/ \kappa = 0.1, 1, 5$ and $10$, are examined, with the aim to support the assertion in Remark~\ref{rem:ikappa}. The following average pointwise error between $\eb_{\mathrm{exct}}$ and $\eb_h$ is calculated
\[
    \mathrm{err}_h := \dfrac{1}{N_S} \sum_{i = 1}^{N_S} \paren{\abs{\eb_h(\xb_i) - \eb_{\mathrm{exct}}(\xb_i)}^2 + \abs{\curlt \, \eb_h(\xb_i) - \curlt \, \eb_{\mathrm{exct}}(\xb_i)}^2}^{1/2},
\]
where $\xb_1, \xb_2, \ldots, \xb_{N_S}$, with $N_S = 5000$, are the evaluation points uniformly distributed on the centered sphere of radius $2\mathrm{m}$. Figure~\ref{fig:scatted_field} \textit{(left)} depicts the average pointwise error with respect to $h$ for different $\kappa^\prime$. On one hand, this result confirms the unique solvability of the CFIE \eqref{eq:cfie} at a resonant frequency as well as the convergence of its Galerkin discretization \eqref{eq:discrete_cfie}, regardless of the imaginary wave number $i\kappa^\prime$. On the other hand, the effect of $\kappa^\prime$ on numerical solutions is minor. When using larger $\kappa^\prime$, the corresponding matrices become to good approximation sparse, but more quadrature points are required to accurately compute the non-zero entries.

It is well-known that the choice of the coupling parameter $\eta$ has a major impact on the conditioning property of the Galerkin discrete system of CFIEs. There is no general theory concerning the optimal value of $\eta$ that yields the lowest condition number. Some discussions on this topic can be found in \cite{KS1983,Kress1985}. In the second experiment, we investigate the impact of the coupling parameter $\eta$ on the spectral condition number of $\GGs^{-\transpose} \LLs$ by considering different $\eta \in \R \setminus \{0\}$ with its absolute value $\abs{\eta}$ ranging from $10^{-4} \kappa^2$ to $10^4 \kappa^2$ (again with $\kappa = 4.4934$). A boundary mesh with meshwidth $h = 0.15\mathrm{m}$ is involved. Figure~\ref{fig:scatted_field} \textit{(right)} shows that the values of $\eta$ around $\pm \kappa^2$ are optimal for the unit sphere. In the following experiments in this section, the coupling parameter $\eta = - \kappa^2$ is chosen.

\begin{figure}
    \centering
    \begin{subfigure}[b]{0.49\textwidth}
        \centering
        \includegraphics[width=\textwidth]{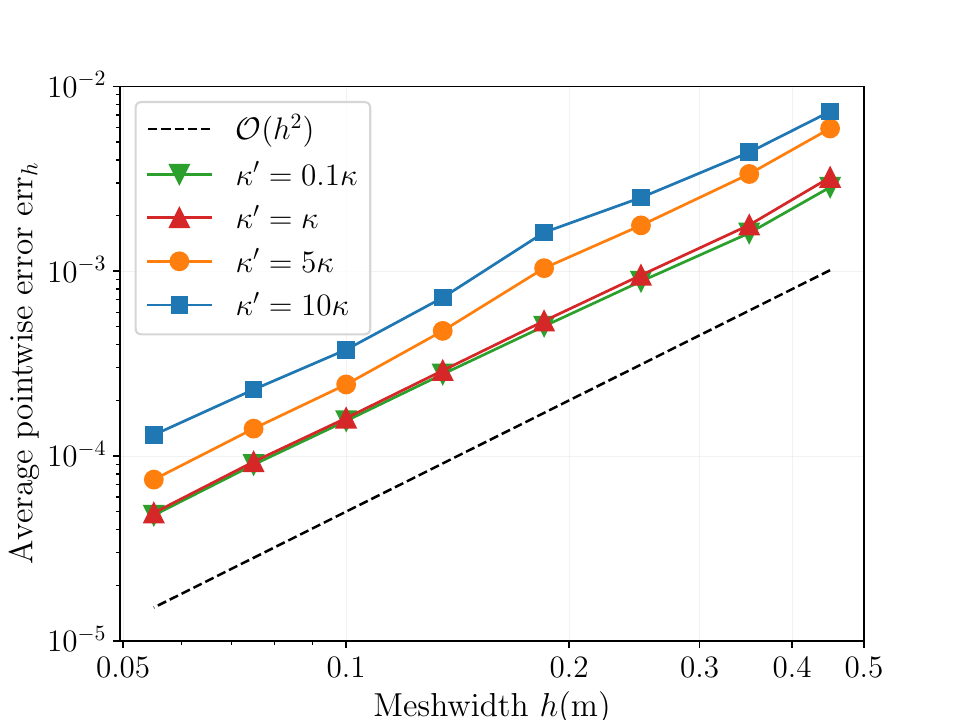}
    \end{subfigure}
    \hfill
    \begin{subfigure}[b]{0.49\textwidth}
        \centering 
        \includegraphics[width=\textwidth]{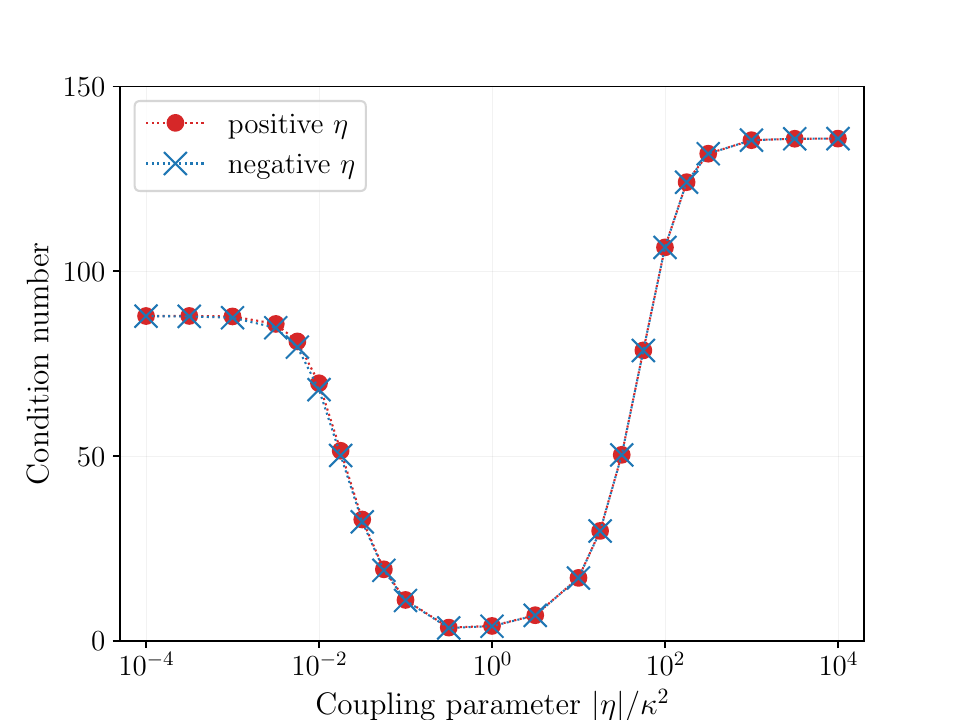}
    \end{subfigure}
    \caption{Average pointwise error of the scattered electric field $\eb$ with respect to meshwidth $h$ for different purely imaginary wave numbers $i\kappa^\prime$ (\textit{left}), and spectral condition number of matrices arising from the Galerkin discretization of the proposed CFIE with respect to coupling parameter $\eta$ (\textit{right}) for a sphere of radius $1\mathrm{m}$.}
    \label{fig:scatted_field}
\end{figure}

Next, we validate the stability of the Galerkin discretization \eqref{eq:discrete_cfie} by examining the condition number of the matrix $\GGs^{-\transpose} \LLs$ with varying meshwidth $h$ and wave number $\kappa$. The corresponding number of GMRES iterations\footnote{ Even though the condition number is only in special cases related to the number of required GMRES iterations, it seems to be a good predictor for the discretization scheme studied here. Establishing a rigorous relation between the condition number and the iteration count is out of the scope of this paper. Some relevant important insights on this topic can be found in \cite{EES1983,GPS1996,EE2001,CG2019}.} required to solve \eqref{eq:discrete_cfie} (with tolerance $\veps_0 = 10^{-8}$) is also reported to demonstrate the practicality of the proposed discretization scheme. For the purpose of comparison, the performance of the standard EFIE discretized by the linear Raviart-Thomas boundary elements is also presented. Firstly, the wave number is fixed at $\kappa = \frac{2\pi}{3}$ (which is far away from resonant wave numbers), and the meshwidth varies from $0.055\mathrm{m}$ to $0.45\mathrm{m}$. Figure~\ref{fig:sphere_cond} (\textit{left}) shows that, whereas the condition number of matrices derived from the EFIE and the corresponding number of iterations increase when the meshwidth decreases, those of the proposed CFIE remain stable. Secondly, the fixed meshwidth $h = 0.15\mathrm{m}$, while the wave number varies from $0.005$ to $4.28$. This range contains two resonant wave numbers $\kappa = 2.7437$ and $\kappa = 3.8702$ (see \cite[Section~7.2]{Jin2010}). As shown in Figure~\ref{fig:sphere_cond} (\textit{right}), the condition number of EFIE's discrete systems and the number of GMRES iterations required to solve those systems grow dramatically when the wave number gets close to the resonant ones or approaches 0. In contrast, the condition number of the discrete linear system \eqref{eq:discrete_cfie} and the corresponding number of iterations stay almost constant.

\begin{figure}
    \centering
    \begin{subfigure}[b]{0.49\textwidth}
        \centering
        \includegraphics[width=\textwidth]{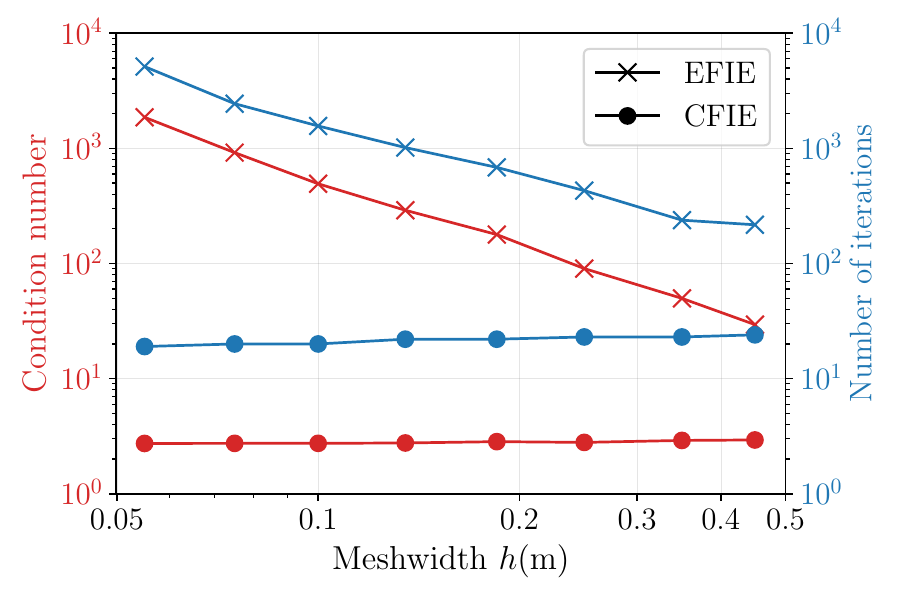}
    \end{subfigure}
    \hfill
    \begin{subfigure}[b]{0.49\textwidth}
        \centering 
        \includegraphics[width=\textwidth]{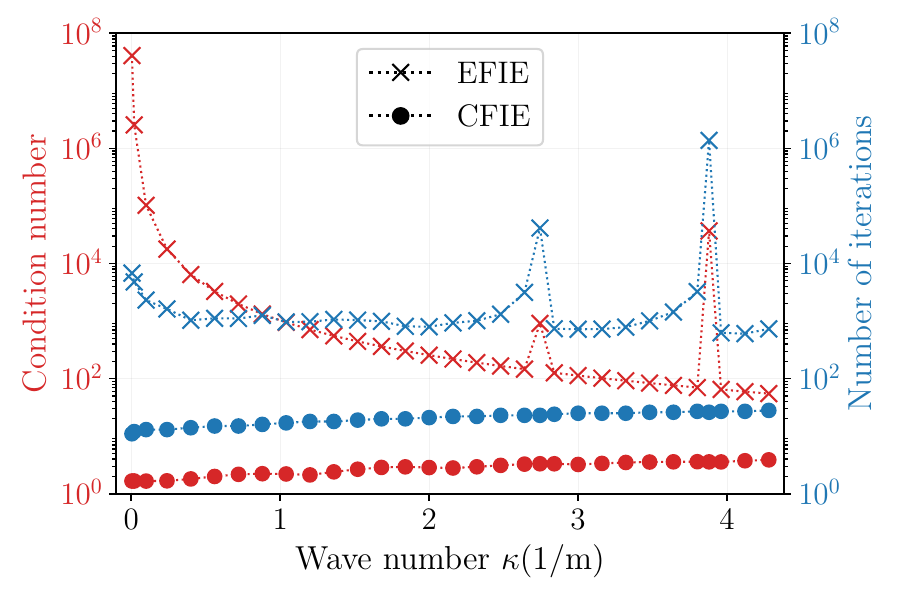}
    \end{subfigure}
    \caption{Condition number of matrices arising from the Galerkin discretization of the EFIE and the proposed CFIE, together with number of GMRES iterations required to solve the corresponding discrete systems for a sphere of radius $1\mathrm{m}$. \textit{Left:} the wave number is fixed at $\kappa = \tfrac{2\pi}{3}$, and the meshwidth is varying. \textit{Right:} the meshwidth is fixed at $h = 0.15\mathrm{m}$, and the wave number is varying.}
    \label{fig:sphere_cond}
\end{figure}

\subsection{Cube}
We now consider the scattering by a simple Lipschitz polyhedron: a cube of edge $1\mathrm{m}$. The resonant wave numbers associated with this domain are determined by $\pi \sqrt{l^2 + m^2 + n^2}$, where $l, m$ and $n$ are non-negative integers of which at most one is zero, see \cite[Section~10.4.2]{VanBladel2007}. We investigate the performance of the proposed Galerkin discretization scheme for the unit cube by repeating two last experiments for the sphere. Firstly, the wave number is fixed at $\kappa = \frac{\pi}{2}$, which is smaller than the smallest resonant wave number, and the meshwidth is ranging from $0.04\mathrm{m}$ to $0.35\mathrm{m}$. Afterwards, the meshwidth is chosen as $h = 0.1\mathrm{m}$, and the wave number varies in the range from $0.01$ to $6$, which contains two resonant wave numbers $\sqrt{2} \pi \approx 4.4429$ and $\sqrt{3} \pi \approx 5.4414$. The condition number of matrices derived from the Galerkin discretization of the standard EFIE and the proposed CFIE, and the corresponding number of GMRES iterations required to solve the discrete systems are depicted in Figure~\ref{fig:cube_cond}. The discretization of EFIE is unstable when the wave number $\kappa$ is close to the resonant ones or 0, or when a fine mesh is used. In both cases, the discrete linear system \eqref{eq:discrete_cfie} remains stable. These results corroborate the solvability of the CFIE \eqref{eq:cfie} for all wave numbers, as well as the well-conditioning of its Galerkin discretization regardless of the resolution to the discrete problem.

\begin{figure}
    \centering
    \begin{subfigure}[b]{0.49\textwidth}
        \centering
        \includegraphics[width=\textwidth]{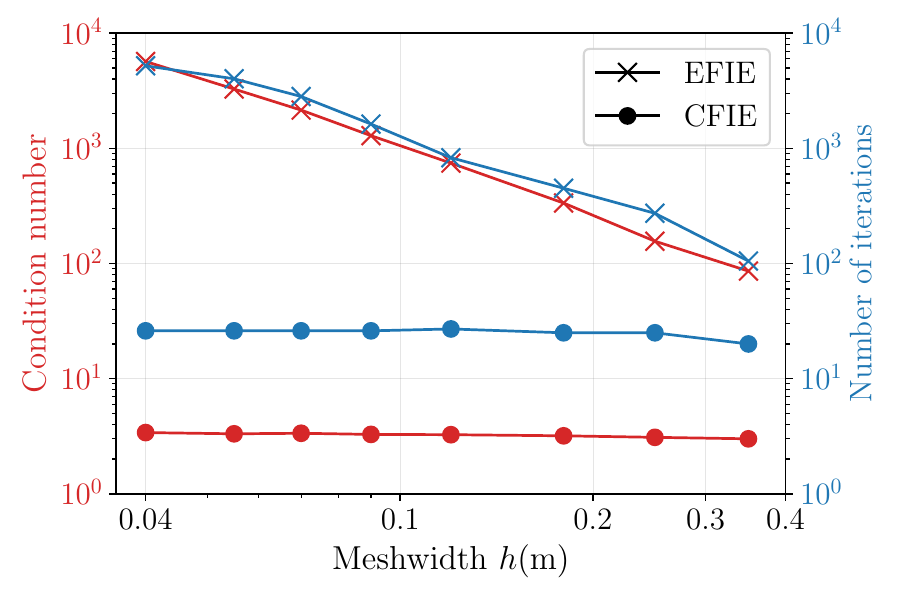}
    \end{subfigure}
    \hfill
    \begin{subfigure}[b]{0.49\textwidth}
        \centering 
        \includegraphics[width=\textwidth]{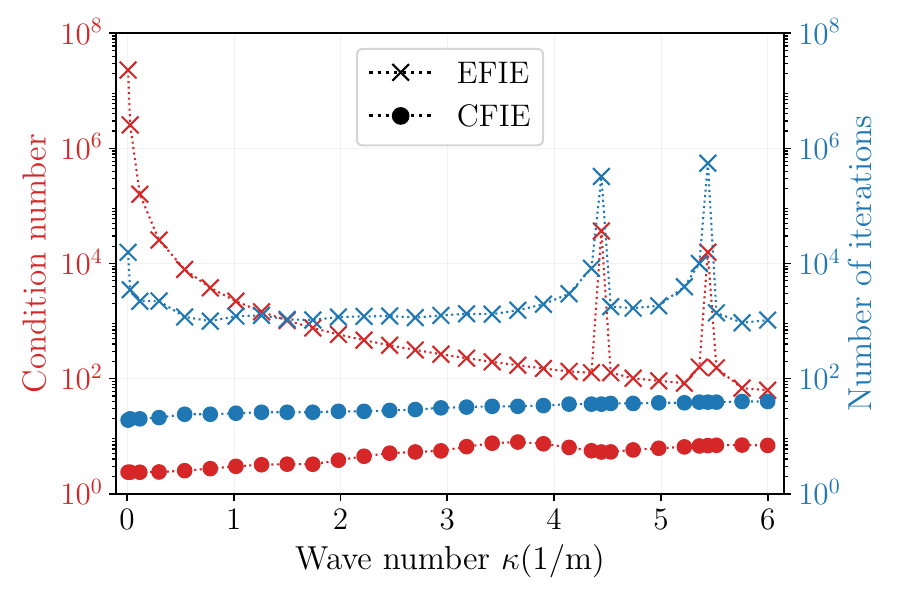}
    \end{subfigure}
    \caption{Condition number of matrices arising from the Galerkin discretization of the EFIE and the proposed CFIE, {together with number of GMRES iterations required to solve the corresponding discrete systems} for a cube of edge $1\mathrm{m}$. \textit{Left:} the wave number is fixed at $\kappa = \frac{\pi}{2}$, and the meshwidth is varying. \textit{Right:} the meshwidth is fixed at $h = 0.1\mathrm{m}$, and the wave number is varying.}
    \label{fig:cube_cond}
\end{figure}

\section{Conclusions}
\label{sec:conclusions}

In this contribution, we have introduced a boundary integral equation for electromagnetic scattering by a perfect electric conductor with Lipschitz continuous boundary. This equation yields a unique solution for all wave numbers. In addition, we have also proposed a Galerkin boundary element discretization for the integral equation, for which the unique solvability, the convergence and the well-conditioning have been demonstrated.

In a forthcoming work, the convergence rate of the proposed discretization scheme should be investigated. To boost the robustness of the numerical scheme, particularly for complex geometries with corners, other pairs of dual meshes and the corresponding boundary elements should also be considered.  


\bibliographystyle{siamplain}

\end{document}

%% file: lvc_style.tex
\newcommand{\R}{\mathbb R}
\newcommand{\C}{\mathbb C}

\newcommand{\EE}{\mathcal E}

\newcommand{\LL}{\mathcal L}

\newcommand{\TT}{\mathcal T}

\DeclareMathOperator{\Ls}{L}
\DeclareMathOperator{\LLs}{\mathbf L}

\DeclareMathOperator{\Hs}{H}
\DeclareMathOperator{\HHs}{\mathbf H}

\DeclareMathOperator{\AAs}{\mathbf A}
\DeclareMathOperator{\DDs}{\mathbf D}
\DeclareMathOperator{\KKs}{\mathbf K}
\DeclareMathOperator{\NNs}{\mathbf N}
\DeclareMathOperator{\Cs}{C}
\DeclareMathOperator{\CCs}{\mathbf C}

\DeclareMathOperator{\UUs}{\mathbf U}
\DeclareMathOperator{\VVs}{\mathbf V}
\DeclareMathOperator{\SSs}{\mathbf S}
\DeclareMathOperator{\XXs}{\mathbf X}
\DeclareMathOperator{\YYs}{\mathbf Y}

\DeclareMathOperator{\Real}{Re}

\DeclareMathOperator{\cond}{cond}

\DeclareMathOperator{\RT}{RT}

\DeclareMathOperator{\MMs}{\mathbf M}
\DeclareMathOperator{\GGs}{\mathbf G}

\DeclareMathOperator{\ZZs}{\mathbf Z}

\newcommand{\brac}[1]{\left\lbrace{#1}\right\rbrace}

\newcommand{\paren}[1]{\left({#1}\right)}
\newcommand{\norm}[1]{\left\lVert{#1}\right\rVert}

\newcommand{\abs}[1]{\left\vert{#1}\right\vert}

\newcommand{\inprod}[1]{\left\langle{#1}\right\rangle}

\newcommand{\wt}[1]{\widetilde{#1}}
\newcommand{\wh}[1]{\widehat{#1}}
\newcommand{\ovl}[1]{\overline{#1}}

\newcommand{\transpose}{\mathsf{T}}

\newcommand{\bm}[1]{\boldsymbol{#1}}

\newcommand{\curlt}{\textbf{\textup{curl}}}
\newcommand{\divt}{\textup{div}}
\newcommand{\gradt}{\textbf{grad}}

\newcommand{\veps}{\varepsilon}

\newcommand{\pa}{\partial}
\newcommand{\gm}{\gamma}
\newcommand{\Gm}{\Gamma}
\newcommand{\om}{\omega}
\newcommand{\Om}{\Omega}
\newcommand{\sm}{\sigma}

\newcommand{\vphi}{\varphi}
\newcommand{\fa}{\forall}
\newcommand{\sst}{\subset}

\newcommand{\Ab}{\bm{A}}

\newcommand{\xb}{\bm{x}}
\newcommand{\yb}{\bm{y}}
\newcommand{\zb}{\bm{z}}

\newcommand{\ab}{\bm{a}}
\newcommand{\eb}{\bm{e}}

\newcommand{\vb}{\bm{v}}
\newcommand{\ub}{\bm{u}}
\newcommand{\wb}{\bm{w}}

\newcommand{\tv}{\textbf{t}}
\newcommand{\bb}{\bm{b}}

\newcommand{\nv}{\textbf{n}}

\newcommand{\zrb}{\bm{0}}
\newcommand{\vphib}{\bm{\varphi}}

\newcommand{\psib}{\bm{\psi}}
\newcommand{\xib}{\bm{\xi}}

\newcommand{\Psib}{\bm{\Psi}}

\newcommand{\ra}{\rightarrow}

\newcommand{\sra}{\rightharpoonup}

\newcommand{\dx}{\,\mathrm{d}\xb}

\newcommand{\ds}{\,\mathrm{d}s}

\newcommand{\q}{\quad}
\newcommand{\qq}{\qquad}
\newcommand{\qqq}{\qquad\quad}
\newcommand{\qqqq}{\qquad\qquad}
\newcommand{\qqqqq}{\qquad\qquad\quad}
\newcommand{\qqqqqq}{\qquad\qquad\qquad}